\documentclass[a4paper, 10pt, twoside]{amsart}

\usepackage{enumitem}
\usepackage{amsmath, wrapfig}
\usepackage{latexsym}
\usepackage{graphicx}
\usepackage[usenames, dvipsnames]{color}
\usepackage[square, sort&compress, authoryear]{natbib}
\usepackage[colorlinks=true,
        raiselinks=true,
        linkcolor=MidnightBlue,
        citecolor=Mahogany,
        urlcolor=ForestGreen,
        pdfauthor=Debasish Chatterjee,
        pdftitle={Maximizing the probability of attaining a target prior to extinction, general state space},
        pdfkeywords={stochastic reachability and avoidance, stochastic optimal control, Markov control processes},
        pdfsubject={Technical Report},
        plainpages=false]{hyperref}

\usepackage[charter]{mathdesign}
\DeclareMathOperator{\erfc}{erfc}
\DeclareMathOperator{\erf}{erf}
\DeclareMathOperator{\sat}{sat}
\DeclareMathOperator{\argmax}{arg\,max}

\definecolor{mred}{rgb}{0.6, 0, 0}
\definecolor{mgreen}{rgb}{0, 0.5, 0}
\definecolor{mblue}{rgb}{0, 0, 0.5}
\definecolor{mcyan}{rgb}{0, 0.5, 0.5}

\newcommand{\R}{\ensuremath{\mathbb{R}}}
\newcommand{\N}{\ensuremath{\mathbb{N}}}
\newcommand{\Nz}{\ensuremath{\mathbb{N}_0}}
\newcommand{\posR}{\ensuremath{\R_{\ge 0}}}

\newcommand{\ra}{\ensuremath{\rightarrow}}

\newcommand{\lra}{\ensuremath{\longrightarrow}}

\newcommand{\fa}{\ensuremath{\forall\,}}
\renewcommand{\le}{\ensuremath{\leqslant}}
\renewcommand{\ge}{\ensuremath{\geqslant}}
\renewcommand{\mapsto}{\ensuremath{\longmapsto}}

\newcommand{\setmin}{\ensuremath{\smallsetminus}}
\newcommand{\corresp}{\ensuremath{\lra\!\!\!\!\!\!\!\!\!\!\to{\;}}}

\newcommand{\EE}{\ensuremath{\mathsf{E}}}
\newcommand{\PP}{\ensuremath{\mathsf{P}}}

\newcommand{\epower}[1]{\ensuremath{\mathrm{e}^{#1}}}
\newcommand{\norm}[1]{\ensuremath{\left\lVert #1 \right\rVert}}
\newcommand{\Borelsigalg}[1]{\ensuremath{\mathfrak{B}\!\left(#1\right)}}

\newcommand{\indic}[1]{\ensuremath{\boldsymbol{1}_{#1}}}
\newcommand{\abs}[1]{\ensuremath{\left\lvert{#1}\right\rvert}}

\newcommand{\mrm}[1]{\ensuremath{\mathrm{#1}}}

\newcommand{\Lp}[1]{\ensuremath{\boldsymbol{L}_{#1}}}

\newcommand{\secref}[1]{\S\ref{#1}}

\newcommand{\sigalg}{\ensuremath{\mathfrak{F}}}

\newcommand{\xz}{\ensuremath{x_0}}

\newcommand{\ol}{\overline}

\renewcommand{\subset}{\ensuremath{\subseteq}}

\newcommand{\mn}{\ensuremath{\wedge}}

\newcommand{\DefEnd}{\hspace{\stretch{1}}{$\Diamond$}}
\newcommand{\AssumptionEnd}{\hspace{\stretch{1}}{$\diamondsuit$}}
\newcommand{\vphi}{\ensuremath{\varphi}}

\newcommand\Let{\mathrel{\mathop:\!\!=}}

\allowdisplaybreaks

\newtheoremstyle{nonum}{8pt}{8pt}{}{}{\itshape}{.}{ }{\thmname{#1}\thmnote{ (\mdseries #3)}}
\theoremstyle{nonum}

\newtheoremstyle{nonumt}{8pt}{8pt}{\slshape}{}{\bfseries}{.}{ }{\thmname{#1}\thmnote{ (\mdseries #3)}}
\theoremstyle{nonumt}

\numberwithin{equation}{section}
\swapnumbers
\newtheoremstyle{dcstyle}{8pt}{8pt}{\slshape}{}{\bfseries}{.}{ }{}

\theoremstyle{dcstyle}
\newtheorem{theorem}[equation]{Theorem}

\newtheorem{lemma}[equation]{Lemma}
\newtheorem{proposition}[equation]{Proposition}

\theoremstyle{definition}
\newtheorem{defn}[equation]{Definition}

\theoremstyle{remark}

\newtheorem{prgr}[equation]{}
\newtheorem{assumption}[equation]{Assumption}

\makeatletter
\def\tagform@#1{\maketag@@@{\ignorespaces#1\unskip\@@italiccorr}}
\makeatother
\title[Maximizing the chance of hitting a target prior to extinction]{Maximizing the probability of attaining a target prior to extinction}
\thanks{This research was partially supported by the Swiss National Science Foundation under grant 200021-122072.}
\author[D.~Chatterjee]{Debasish Chatterjee}
\address{Automatic Control Laboratory\\ ETL I19, ETH Z\"urich\\ Physikstrasse 3\\ 8092 Z\"urich\\ Switzerland}
\email{chatterjee@control.ee.ethz.ch}
\urladdr{\url{http://control.ee.ethz.ch/~chatterd}}

\author[E.~Cinquemani]{Eugenio Cinquemani}
\address{INRIA-Grenoble - Rh\^one-Alpes\\ 655 avenue de l'Europe\\ Montbonnot\\ 38334 Saint Ismier cedex\\ France}
\email{Eugenio.Cinquemani@inria.fr}
\urladdr{\url{http://ibis.inrialpes.fr/article941.html}}

\author[J.~Lygeros]{John Lygeros}
\address{Automatic Control Laboratory\\ ETL I22, ETH Z\"urich\\ Physikstrasse 3\\ 8092 Z\"urich\\ Switzerland}
\email{lygeros@control.ee.ethz.ch}
\urladdr{\url{http://control.ee.ethz.ch/~lygeros}}

\date{\today}

\begin{document}

	\begin{abstract}
		We present a dynamic programming-based solution to the problem of maximizing the probability of attaining a target set before hitting a cemetery set for a discrete-time Markov control process. Under mild hypotheses we establish that there exists a deterministic stationary policy that achieves the maximum value of this probability. We demonstrate how the maximization of this probability can be computed through the maximization of an expected total reward until the first hitting time to either the target or the cemetery set. Martingale characterizations of thrifty, equalizing, and optimal policies in the context of our problem are also established.
	\end{abstract}


	\maketitle

	\section{Introduction}
	\label{s:intro}
		There are two basic categories of discrete-time controlled Markov processes that deal with random temporal horizons. The first is the well-known \textsl{optimal stopping problem}~\citep{ref:dynkinOptStopping}, in which the random horizon arises from some dynamic optimization protocol based on the past history of the process. The random `stopping time' thus generated is regarded as a decision variable. This problem arises in, among other areas, stochastic analysis, mathematical statistics, mathematical finance, and financial engineering; see the comprehensive monograph~\citep{ref:shiryaevOptStop} for details and further references. The second is relatively less common, and is characterized by the fact that the random horizon arises as a result of an endogenous event of the stochastic process, e.g., the process hitting a particular subset of the state-space, variations in the process paths crossing a certain threshold. This problem arises in, among others, optimization of target-level criteria~\citep{ref:dubinsHowToGamble, ref:bouakizTargetHitting}, optimal control of retirement investment funds~\citep{ref:boda04}, minimization of ruin probabilities in insurance funds~\citep{ref:schmidliInsurance}, `satisfaction of needs' problems in economics~\citep{ref:simonModelsofMan}, risk minimizing stopping problems~\citep{ref:ohtsuboValueIter}, attainability problems under stochastic perturbations~\citep{ref:kurzhanskiiAttainability}, and optimal control of Markov control processes up to an exit time~\citep{ref:borkarTopicsControlledMC}.

		The problem treated in this article falls under the second category above. In broad strokes, we consider a discrete-time Markov control process with Borel state and action spaces. We assume that there is a certain target set located inside a safe region, the latter being a subset of the state-space. The problem is to maximize the probability of attaining the target set before exiting the safe set (or equivalently, hitting the cemetery set or unsafe region). This `reach a good set while avoiding a bad set' formulation arises in, e.g., air traffic control, where aircraft try to reach their destination while avoiding collision with other aircraft or the ground despite uncertain weather conditions. It also arises in portfolio optimization, where it is desired to reach a target level of wealth without falling below a certain baseline capital with high probability. Finally, it forms the core of the computation of safe sets for hybrid systems where the `good' and the `bad' sets represent states from which a discrete transition into the unsafe set is possible~\citep{ref:lygerosUncertainHyb, ref:tomlin2000}. Special cases of this problem have been investigated in, e.g.,~\citep{ref:lygeros03, ref:prandini06} in the context of air traffic applications,~\citep{ref:prandiniPSafety, ref:prajna07} in the context of probabilistic safety,~\citep{ref:boda04} in the context of maximizing the probability of attaining a preassigned comfort level of retirement investment funds.

		It is clear from the description of our problem in the preceding paragraph that there are two random times involved, namely, the hitting times of the target and the cemetery sets. In this article we formulate our problem as the maximization of an expected total reward accumulated up to the minimum of these two hitting times. As such, this formulation falls under the broad framework of optimal control of Markov control processes up to an exit time, which has a long and rich history. It has mostly been studied as the minimization of an expected total cost until the first time that the state enters a given target set, see e.g.,~\cite[Chapter~II]{ref:borkarTopicsControlledMC}, \cite[Chapter~8]{ref:hernandez-lerma2}, and the references therein. In particular, if a unit cost is incurred as long as the state is outside the target set, then the problem of minimizing the cost accumulated until the state enters the target is known variously as the \textsl{pursuit problem}~\citep{ref:eatonzadeh62}, \textsl{transient programming}~\citep{ref:whittleOptimization}, the \textsl{first passage problem}~\citep{ref:dermanMDP, ref:kushnerIntroStochControl}, the \textsl{stochastic shortest path problem}~\citep{ref:bertsekasDP2}, and \textsl{control up to an exit time}~\citep{ref:borkarConvexAnalyticApproach, ref:borkarTopicsControlledMC, ref:kestenMCP}. Here we exploit certain additional structures of our problem in the dynamic programming equations that we derive leading to methods fine-tuned to the particular problem at hand.

		Our main results center around the assertion that there exists a deterministic stationary policy that maximizes the probability of hitting the target set before the cemetery set. This maximal probability as a function of the initial state is the optimal value function for our problem. We obtain a Bellman equation for our problem which is solved by the optimal value function. Furthermore, we provide martingale-theoretic conditions characterizing `thrifty', `equalizing', and optimal policies via methods derived from~\citep{ref:dubinsHowToGamble, ref:karatzasDP}; see also~\citep{ref:guoSemimartingaleMCP} and the references therein for martingale characterization of average optimality. The principal techniques employed in this article are similar to the ones in~\citep{ref:recstrat}, where the authors studied optimal control of a Markov control process up its first entry time to a safe set.  In~\citep{ref:recstrat} we developed a recovery strategy to enter a given target set from its exterior while minimizing a discounted cost. The problem was posed as one of minimizing the sum of a discounted cost-per-stage function $c$ up to the first entry time $\tau$ to a target set, namely, minimize $\EE^\pi_x\bigl[\sum_{t=0}^{\tau-1}\alpha^t c(x_t, a_t)\bigr]$ over a class of admissible policies $\pi$, where $\alpha\in\;]0,1[$ is a discount factor. Here we extend this approach to problems with two sets, a target and a cemetery, and the case of $\alpha = 1$.

		This article unfolds as follows. The main results are stated in~\secref{s:results}. In~\secref{s:prelims} we define the general setting of the problem, namely, Markov control processes on Polish spaces, their transition kernels, and the admissible control strategies. In~\secref{s:mainres} we present our main Theorem~\ref{t:exist} which guarantees the existence of a deterministic stationary policy that leads to the maximal probability of hitting the target set while avoiding the specified dangerous set, and also provides a Bellman equation that the value function must satisfy. In~\secref{s:martchar} we look at a martingale characterization of the optimal control problem; thrifty and equalizing policies are defined in the context of our problem, and we establish necessary and sufficient conditions for optimality in terms of thrifty and equalizing policies in Theorem~\ref{t:martcharpolicy}. We discuss related reward-per-stage functions and their relationships to our problem and treat several examples in~\secref{s:disc}. Proofs of the main results appear in~\secref{s:proof}. The article concludes in~\secref{s:concl} with a discussion of future work.

	\section{Main Results}
	\label{s:results}
	Our main results are stated in this section after some preliminary definitions and conventions.
	\subsection{Preliminaries}
	\label{s:prelims}
		We employ the following standard notations. Let $\N$ denote the natural numbers $\{1, 2, \ldots\}$ and $\Nz$ denote the nonnegative integers $\{0\}\cup\N$. Let $\indic{A}(\cdot)$ be the usual indicator function of a set $A$, i.e., $\indic{A}(\xi) = 1$ if $\xi\in A$ and $0$ otherwise. For real numbers $a$ and $b$ let $a\mn b \Let  \min\{a, b\}$. A function $f:X\lra\R$ restricted to $A\subset X$ is depicted as $f|_A$.

		Given a nonempty Borel set $X$ (i.e., a Borel subset of a Polish space), its Borel $\sigma$-algebra is denoted by $\Borelsigalg{X}$. By convention, when referring to sets or functions, ``measurable'' means ``Borel-measurable.'' If $X$ and $Y$ are nonempty Borel spaces, a \emph{stochastic kernel} on $X$ given $Y$ is a function $Q(\cdot|\cdot)$ such that $Q(\cdot|y)$ is a probability measure on $X$ for each fixed $y\in Y$, and $Q(B|\cdot)$ is a measurable function on $Y$ for each fixed $B\in\Borelsigalg X$.

		We briefly recall some standard definitions below, see, e.g.,~\citep{ref:hernandez-lerma1} for further details. A \emph{Markov control model} is a five-tuple
		\begin{equation}
			\label{e:mmodel}
			\bigl(X, A, \{A(x)\mid x\in X\}, Q, r\bigr)
		\end{equation}
		consisting of a nonempty Borel space $X$ called the \emph{state-space}, a nonempty Borel space $A$ called the \emph{control} or \emph{action set}, a family $\{A(x)\mid x\in X\}$ of nonempty measurable subsets $A(x)$ of $A$, where $A(x)$ denotes the set of \emph{feasible controls} or \emph{actions} when the system is in state $x\in X$ and with the property that the set $\mathbb K \Let  \bigl\{(x, a)\big|x\in X, a\in A(x)\bigr\}$ of feasible state-action pairs is a measurable subset of $X\times A$, a stochastic kernel $Q$ on $X$ given $\mathbb K$ called the \emph{transition law}, and a measurable function $r:\mathbb K\lra \R$ called the \emph{reward-per-stage function}.


		\begin{assumption}
		\label{a:basic}
		{\rm 
			The set $\mathbb K$ of feasible state-action pairs contains the graph of a measurable function from $X$ to $A$.
		}\AssumptionEnd
		\end{assumption}

		Consider the Markov model~\eqref{e:mmodel}, and for each $i=0, 1, \ldots,$ define the space $H_i$ of \emph{admissible histories} up to time $i$ as $H_0 \Let  X$ and $H_i \Let  \mathbb K^i\times X = \mathbb K\times H_{i-1}, i\in \N$. A generic element $h_i$ of $H_i$, which is called an admissible $i$-history, or simply $i$-history, is a vector of the form $h_i = (x_0, a_0, \ldots, x_{i-1}, a_{i-1}, x_i)$, with $(x_j, a_j)\in\mathbb K$ for $j=0, \ldots, i-1$, and $x_i\in X$. 
		Hereafter we let the $\sigma$-algebra generated by the history $h_i$ be denoted by $\sigalg_i$, $i\in\Nz$.

		Recall that a \emph{policy} is a sequence $\pi = (\pi_i)_{i\in\Nz}$ of stochastic kernels $\pi_i$ on the control set $A$ given $H_i$ satisfying the constraint $\pi_i(A(x_i)|h_i) = 1\;\;\fa h_i\in H_i, i\in\Nz$. The set of all policies is denoted by $\Pi$. Let $(\Omega, \sigalg)$ be the measurable space consisting of the (canonical) sample space $\Omega \Let  \ol H_\infty = (X\times A)^\infty$ and let $\sigalg$ be the corresponding product $\sigma$-algebra. The elements of $\Omega$ are sequences of the form $\omega = (x_0, a_0, x_1, a_1, \ldots)$ with $x_i\in X$ and $a_i\in A$ for all $i\in\Nz$; the projections $x_i$ and $a_i$ from $\Omega$ to the sets $X$ and $A$ are called \emph{state} and \emph{control} (or \emph{action}) variables, respectively.

		Let $\pi = (\pi_i)_{i\in\Nz}$ be an arbitrary control policy, and let $\nu$ be an arbitrary probability measure on $X$, referred to as the initial distribution. By a theorem of Ionescu-Tulcea~\cite[Chapter 3, \S4, Theorem~5]{ref:raoProbTheo}, there exists a unique probability measure $\mathsf P_\nu^\pi$ on $(\Omega, \sigalg)$ supported on $H^\infty$, such that for all $B\in\Borelsigalg X$, $C\in\Borelsigalg A$, $h_i\in H_i$, $i\in\Nz$, we have $\mathsf P_\nu^\pi(\xz\in B) = \nu(B)$,
		\begin{subequations}
		\label{e:probmeasure}
		\begin{align}
			\label{e:actiontrans}
			\mathsf P_\nu^\pi\bigl(a_i\in C\,\big|\, h_i\bigr) & = \pi_i\bigl(C\,\big|\, h_i\bigr)\\
			\label{e:statetrans}
			\mathsf P_\nu^\pi\bigl(x_{i+1}\in B\,\big|\, h_i, a_i\bigr) & = Q\bigl(B\,\big|\, x_i, a_i\bigr).
		\end{align}
		\end{subequations}

		\begin{defn}
		\label{d:mcp}
			{\rm 
			The stochastic process $\bigl(\Omega, \sigalg, \mathsf P_\nu^\pi, (x_i)_{i\in\Nz}\bigr)$ is called a discrete-time \emph{Markov control process}.
			}\DefEnd
		\end{defn}

		We note that the Markov control process in Definition~\ref{d:mcp} is not necessarily Markovian in the usual sense due to the dependence on the entire history $h_i$ in~\eqref{e:actiontrans}; however, it is well-known~\cite[Proposition~2.3.5]{ref:hernandez-lerma1} that if $(\pi_i)_{i\in\Nz}$ is restricted to a suitable subclass of policies, then $(x_i)_{i\in\Nz}$ is a Markov process.

		Let $\Phi$ denote the set of stochastic kernels $\vphi$ on $A$ given $X$ such that $\vphi(A(x)| x) = 1$ for all $x\in X$, and let $\mathbb F$ denote the set of all measurable functions $f:X\lra A$ satisfying $f(x)\in A(x)$ for all $x\in X$. The functions in $\mathbb F$ are called \emph{measurable selectors} of the set-valued mapping $X\ni x\mapsto A(x)\subset A$. Recall that a policy $\pi = (\pi_i)_{i\in\Nz}\in\Pi$ is said to be \emph{randomized Markov} if there exists a sequence $(\vphi_i)_{i\in\Nz}$ of stochastic kernels $\vphi_i\in\Phi$ such that $\pi_i(\cdot| h_i) = \vphi_i(\cdot| x_i)\;\; \fa h_i\in H_i, \; i\in\Nz$; \emph{deterministic Markov} if there exists a sequence $(f_i)_{i\in\Nz}$ of functions $f_i\in\mathbb F$ such that $\pi_i(\cdot| h_i) = \delta_{f(x_i)}(\cdot)$; \emph{deterministic stationary} if there exists a function $f\in\mathbb F$ such that $\pi_i(\cdot| h_i) = \delta_{f(x_i)}(\cdot)$. As usual let $\Pi$, $\Pi_{RM}$, $\Pi_{DM}$, and $\Pi_{DS}$ denote the set of all randomized history-dependent, randomized Markov, deterministic Markov, and deterministic stationary policies, respectively. The transition kernel $Q$ in~\eqref{e:statetrans} under a policy $\pi \Let  (\vphi_i)_{i\in\Nz}\in\Pi_{RM}$ is given by $\bigl(Q(\cdot|\cdot, \vphi_i)\bigr)_{i\in\Nz}$, which is defined as the transition kernel $\Borelsigalg{X}\times X\ni (B, x)\mapsto Q(B|x, \vphi_i(x)) \Let  \int_{A(x)}\vphi_i(\mrm da|x) Q(B|x, a)$. Occasionally we suppress the dependence of $\vphi_i$ on $x$ and write $Q(B|x, \vphi_i)$ in place of $Q(B|x, \vphi_i(x))$, and $r(x_j, \vphi_j) \Let  \int_{A(x_j)} \vphi_j(\mrm da|x_j)r(x_j, a)$. We simply write $f^\infty$ for a policy $(f, f, \ldots)\in \Pi_{DS}$.

	\subsection{Maximizing the Probability of Hitting a Target before a Cemetery Set}
	\label{s:mainres}
		Let $O$ and $K$ be two nonempty measurable subsets of $X$ with $O\subsetneqq K$. Let 
		\begin{align}
		\label{e:taudef}
			\tau \Let  \inf\bigl\{t \in\Nz \;\big|\; x_t\in O\bigr\} \quad\text{and}\quad \tau' \Let  \inf\bigl\{t \in\Nz \;\big|\; x_t\in X\setmin K\bigr\}
		\end{align}
		be the first hitting times of the above sets.\footnote{As usual we set the infimum over an empty set to be $\infty$.} These random times are stopping times with respect to the filtration $(\sigalg_n)_{n\in\Nz}$. Suppose that the objective is to maximize the probability that the state hits the set $O$ before exiting the set $K$; in symbols the objective is to attain
		\begin{equation}
		\label{e:problem}
			V^\star(x) \Let \sup_{\pi\in\Pi} V(\pi, x) \Let  \sup_{\pi\in\Pi}\PP^\pi_x\bigl(\tau < \tau', \tau < \infty\bigr),
		\end{equation}
		where the $\sup$ is taken over a class $\Pi$ of admissible policies.

		\begin{prgr}
		\label{pgr:policies}
		{\rm 
			\emph{Admissible policies.} It is clear at once that the class of admissible policies for the problem~\eqref{e:problem} is different from the classes considered in~\secref{s:prelims}. Indeed, since the process is killed at the stopping time $\tau\mn\tau'$, it follows that the class of admissible policies should also be truncated at the stage $\tau\mn\tau'-1$. For a given stage $t\in\Nz$ we define the $t$-th policy element $\pi_t$ only on the set $\{t < \tau\mn\tau'\}$. Note that with this definition $\pi_t$ becomes a $\sigalg_{t\mn\tau\mn\tau'}$-measurable randomized control. It is also immediate from the definitions of $\tau$ and $\tau'$ that if the initial condition $x\in O\cup(X\setmin K)$, then the set of admissible policies is empty in the sense that there is nothing to do by definition. Indeed, in this case $\tau\mn\tau' = 0$ and no control is needed. We are thus interested only in $x\in K\setmin O$, for otherwise the problem is trivial. In other words, the domain of $\pi_t$ is contained in the `spatial' region $\bigl\{(x, a)\in\mathbb K\,\big|\,x\in K\setmin O, a\in A(x)\bigr\}$. Equivalently, in view of the definitions of the `temporal' elements $\tau$ and $\tau'$, $\pi_t$ is well-defined on the set $\{t < \tau\mn\tau'\}$. We re-define $\mathbb K \Let  \bigl\{(x, a)\in\mathbb K\,\big|\,x\in K\setmin O, a\in A(x)\bigr\}$, and also let $\mathbb F$ to be the set of measurable selectors of the set-valued map $K\setmin O\ni x\mapsto A(x)\subset A$.
		}
		\end{prgr}

		\emph{Throughout this subsection we shall denote by $\Pi_M$ the class of Markov policies such that if $(\pi_t)_{t\in\Nz}\in\Pi_M$, then $\pi_t$ is defined on $\mathbb K$ for each $t$.}
		\medskip

		\begin{prgr}
		\label{pgr:recalldef}
		{\rm 
			Recall that a transition kernel $Q$ on a measurable space $X$ given another measurable space $Y$ is said to be \emph{strongly Feller} if the mapping $y\mapsto \int_X g(x) Q(\mrm dx| y)$ is continuous and bounded for every measurable and bounded function $g:X\lra\R$. A function $g:\mathbb K\lra\R$ is \emph{upper semicontinuous} (u.s.c.)\ if for every sequence $(x_j, a_j)_{j\in\N}\subset\mathbb K$ converging to $(x, a)\in\mathbb K$, we have $\limsup_{j\ra\infty} g(x_j, a_j) \le g(x, a)$; or, equivalently, if for every $r\in\R$, the set $\bigl\{(x, a)\in\mathbb K\,\big|\, g(x, a) \ge r\bigr\}$ is closed in $\mathbb K$. A set-valued map $\Psi:X\corresp Y$ between topological spaces is \emph{upper hemicontinuous at a point $x$} if for every neighborhood $U$ of $\Psi(x)$ there exists a neighborhood $V$ of $x$ such that $z\in V$ implies that $\Psi(z)\subset U$; $\Psi$ is \emph{upper hemicontinuous} if it is upper hemicontinuous at every $x$ in its domain. If $X$ is equipped with a $\sigma$-algebra $\Sigma$, then the set-valued map $\Psi$ is called \emph{weakly measurable} if $\Psi^{\ell}(G)\in\Sigma$ for every open $G\subset Y$, where $\Psi^{\ell}$ is the lower inverse of $\Psi$, defined by $\Psi^{\ell}(A) \Let  \{x\in X\mid \Psi(x)\cap A\neq\emptyset\}$. See, e.g.,~\cite[Chapters~17-18]{ref:aliprantisIDA} for further details on set-valued maps.\footnote{What we call ``set-valued maps'' are ``correspondences'' in~\citep{ref:aliprantisIDA}.} Whenever $B\subset X$ is a nonempty measurable set and we are concerned with any set-valued map $B\ni x\mapsto A(x)\subset A$, we let $B$ be equipped with the trace of $\Borelsigalg X$ on $B$. Let $b\Borelsigalg{X}^+$ denote the convex cone of nonnegative, bounded, and measurable real-valued functions on $X$, and we define $\bar B \Let  \bigl\{g\in\Lp\infty(X)\,\big|\,g|_{X\setmin K} = 0, \norm{g}_{\Lp\infty{(X)}} \le 1\bigr\}$.
		}
		\end{prgr}

		\begin{assumption}
			\label{a:key}
			{\rm 
			In addition to Assumption~\ref{a:basic}, we stipulate that
			\begin{enumerate}[align=right, leftmargin=*, widest=iii, label=(\roman*)]
				\item the set-valued map $K\setmin O\ni x\mapsto A(x)\subset A$ is compact-valued, upper hemicontinuous, and weakly measurable;
				\item the transition kernel $Q$ on $X$ given $\mathbb K$ is strongly Feller, i.e., the mapping $\mathbb K\ni(x, a)\mapsto \int_X Q(\mrm dy|x, a) g(y)$ is continuous and bounded for all bounded and measurable functions $g:X\lra\R$.\AssumptionEnd
			\end{enumerate}
			}
		\end{assumption}

		The following theorem gives basic existence results for the problem~\eqref{e:problem}; a proof is presented in~\secref{s:mainproof}.


		\begin{theorem}
		\label{t:exist}
			Suppose that Assumption~\eqref{a:key} holds, and that $\tau\mn\tau'$ is finite for every policy in $\Pi_M$. Then:
			\begin{enumerate}[label={\rm (\roman*)}, leftmargin=*, widest=iii, align=right]
				\item The value function $V^\star$ is a pointwise bounded and measurable solution to the \emph{Bellman equation} in $\psi$:
				\begin{equation}
				\label{e:bellmaneqn}
					\psi(x) = \indic{O}(x) + \indic{K\setmin O}(x) \max_{a\in A(x)}\int_X Q(\mrm dy|x, a)\indic{K}(y)\psi(y)\quad \fa x\in X.
				\end{equation}
				Moreover, $V^\star$ is minimal in $\bar B\cap b\Borelsigalg{X}^+$.
				\item There exists a measurable selector $f_\star\in\mathbb F$ such that $f_\star(x)\in A(x)$ attains the maximum in~\eqref{e:bellmaneqn} for each $x\in K\setmin O$, which satisfies\label{co:exist:2}
				\begin{equation}
				\label{e:selectorcond}
					V^\star(x) = 
					\begin{cases}
						1 & \text{if } x\in O,\\
						\displaystyle{\int_{K} Q(\mrm dy|x, f_\star)\,V^\star(y)} & \text{if } x\in K\setmin O,\\
						0 & \text{otherwise},
					\end{cases}
				\end{equation}
				where $V^\star$ is as defined in~\eqref{e:Vstar}. Moreover, the deterministic stationary policy $f_\star^\infty$ is optimal. Conversely, if $f_\star^\infty$ is optimal, then it satisfies~\eqref{e:selectorcond}.
			\end{enumerate}
		\end{theorem}


		\begin{prgr}
		\label{pgr:altrep}
		{\rm 
			As a matter of notation we shall henceforth represent the functional equation~\eqref{e:selectorcond} with the less formal version:
			\begin{equation}
			\label{e:selectorcondalt}
				V^\star(x) = \indic{O}(x) + \indic{K\setmin O}(x)\int_{K}Q(\mrm dy|x, f_\star)\,V^\star(y)\quad \fa x\in X.
			\end{equation}
			Note that the measure $Q(\cdot|x, f_\star)$ is not well-defined for $x\in O\cup(X\setmin K)$ for $f\in \mathbb F$ in view of the definition in paragraph~\ref{pgr:policies}. As such the integral $\int_K Q(\mrm dy|x, f_\star)\,V^\star(y)$ is undefined for $x\in O\cup(X\setmin K)$. However, to preserve the form of~\eqref{e:bellmaneqn} and simplify notation, we shall stick to the representation~\eqref{e:selectorcondalt} by defining any object that is written as an integral of a bounded measurable function with respect to the measure $Q(\cdot|x, f)$ to be $0$ whenever $x\in O\cup(X\setmin K)$ and $f \in\mathbb F$.
		}
		\end{prgr}


	\subsection{A Martingale Characterization}
	\label{s:martchar}
		\emph{We now return to the more general class of all possible policies (not just Markovian), denoted by $\Pi$.}

		Fix an initial state $x\in X$ and a policy $\pi\in\Pi$. For each $n\in\N$ we define the random variable $W_n(\pi, x) \Let  \sum_{t=0}^{(n-1)\mn\tau\mn\tau'}\indic{O}(x_t)$. Let us consider the process $(\zeta_n)_{n\in\Nz}$ defined by
		\begin{equation}
		\label{e:zetadef}
		\begin{aligned}
			\zeta_0 & \Let  V^\star(x),\\
			\zeta_n & \Let  W_n(\pi, x) + \indic{K\setmin O}(x_{(n-1)\mn\tau\mn\tau'})(\indic{K}\cdot V^\star)(x_{n\mn\tau\mn\tau'}),\;\;n\in\N.
		\end{aligned}
		\end{equation}

		We follow the basic framework of~\citep{ref:karatzasDP}.
		\begin{defn}
		{\rm 
			A policy $\pi\in\Pi$ is called \emph{thrifty at $x\in X$} if $V^\star(x) = \Lambda^\pi(x)$, and \emph{equalizing at $x\in X$} if $\Lambda^\pi(x) = V(\pi, x)$. The action $a_n$, defined on $\{\tau\mn\tau' > n\}$, is said to \emph{conserve $V^\star$ at $x_n$} if $\indic{O}(x_n) + \indic{K\setmin O}(x_n)\int_K Q(\mrm dy|x_n, a_n) V^\star(y) = V^\star(x_n)$.
		}\DefEnd
		\end{defn}

		Connections between equalizing, thrifty, and optimal policies for our problem~\eqref{e:problem} are established by the following


		\begin{theorem}
		\label{t:martcharpolicy}
			A policy $\pi\in\Pi$ is 
			\begin{itemize}[label=$\circ$, leftmargin=*]
				\item equalizing at $x\in X$ if and only if 
				\[
					\lim_{n\to\infty}\EE^\pi_x\bigl[\indic{K\setmin O}(x_{(n-1)\mn\tau\mn\tau'})(\indic{K}V^\star)(x_{n\mn\tau\mn\tau'})\bigr] = 0;
				\]
				\item optimal at $x\in X$ if and only if $\pi$ is both thrifty and equalizing.
			\end{itemize}
		\end{theorem}


		A connection between thrifty policies, the process $(\zeta_n)_{n\in\Nz}$ defined in~\eqref{e:zetadef}, and actions conserving the optimal value function $V^\star$ is established by the following


		\begin{theorem}
		\label{t:thriftychar}
			For a given policy $\pi\in\Pi$ and an initial state $x\in X$ the following are equivalent:
			\begin{enumerate}[label={\rm (\roman*)}, leftmargin=*, align=right, widest=iii]
				\item $\pi$ is trifty at $x$;
				\item $(\zeta_n)_{n\in\Nz}$ is a $(\sigalg_n)_{n\in\Nz}$ -martingale under $\PP^\pi_x$;
				\item $\PP^\pi_x$-almost everywhere on $\{\tau\mn\tau' > n\}$ the action $a_n$ conserves $V^\star$.
			\end{enumerate}
		\end{theorem}


		It is possible to make a connection, relying purely on martingale-theoretic arguments, between the process $(\zeta_n)_{n\in\Nz}$ and the value function corresponding to an optimal policy. This is the content of the following theorem, which may be viewed as a partial converse to Theorem~\ref{t:thriftychar}.


		\begin{theorem}
		\label{t:Vprimechar}
			Suppose that either one of the stopping times $\tau$ and $\tau'$ defined in~\eqref{e:taudef} is finite for every policy in $\Pi$. Let $V'$ be a nonnegative measurable function such that $V'|_O = 1$, $V'|_{X\setmin K} = 0$, and bounded above by $1$ elsewhere. For a policy $\pi\in\Pi$ define the process $(\zeta'_n)_{n\in\Nz}$ as
			\begin{equation}
			\label{e:zetapdef}
			\begin{aligned}
				\zeta'_0 & \Let  V'(x),\\
				\zeta'_n & \Let  W_n(\pi, x) + \indic{K\setmin O}(x_{(n-1)\mn\tau\mn\tau'})(\indic{K}\cdot V')(x_{n\mn\tau\mn\tau'}),\;\;n\in\N,
			\end{aligned}
			\end{equation}
			where $W_n(\pi, x)$ is as in~\eqref{e:zetadef}. If for some policy $\pi^\star\in\Pi$ the process $(\zeta'_n)_{n\in\Nz}$ is a $(\sigalg_n)_{n\in\Nz}$ -martingale under $\PP^{\pi^\star}_x$, then $V'(x) = \PP^{\pi^\star}_x\bigl(\tau < \tau', \tau < \infty\bigr)$.
		\end{theorem}


		Proofs of the above results are presented in~\secref{s:martproofs}.

	\section{Discussion and Examples}
		\label{s:disc}
		Let us look at the stopped process $(x_{t\mn{(n-1)}\mn\tau\mn\tau'})_{t\in\Nz}$. It is clear that in this case $V_n(\pi, x) = 1$ whenever $x\in O$ and $V_n(\pi, x) = 0$ whenever $x\in X\setmin K$ for all policies in $\Pi_M$; otherwise for $x\in K\setmin O$ we have
		\begin{align*}
			V_n(\pi, x) & \Let  \PP^\pi_x\bigl(\tau < \tau', \tau < n\bigr)\\
				& = \PP^\pi_x\bigl(x_{1\mn\tau\mn\tau'}\in O\bigr) + \PP^\pi_x\bigl(x_{1\mn\tau\mn\tau'}\in K\setmin O, x_{2\mn\tau\mn\tau'}\in O\bigr)\\
				& \quad + \ldots + \PP^\pi_x\bigl(x_{1\mn\tau\mn\tau'}, \ldots, x_{(n-2)\mn\tau\mn\tau'}\in K\setmin O, x_{(n-1)\mn\tau\mn\tau'}\in O\bigr).
		\end{align*}
		Since the $k$-th term on the right-hand side is 
		$\EE^\pi_x\bigl[\prod_{t=1}^{k-1}\indic{K\setmin O}(x_{t\mn\tau\mn\tau'})\indic{O}(x_{k\mn\tau\mn\tau'})\bigr]$, it follows that
		\begin{align*}
				V_n(\pi, x) & = \EE^\pi_x\Biggl[\sum_{t=1}^{(n-1)\mn\tau\mn\tau'} \Biggl(\prod_{i=0}^{t-1}\indic{K\setmin O}(x_i)\Biggr)\indic{O}(x_t)\Biggr]\\
				& = \EE^\pi_x\Biggl[\sum_{t=1}^{(n-1)\mn\tau} \indic{O}(x_{t\mn\tau'})\Biggr] = \EE^\pi_x\Biggl[\sum_{t=1}^{(n-1)\mn\tau\mn\tau'} \indic{O}(x_{t})\Biggr].
		\end{align*}
		We note that $V_n(\pi, x) = 0$ whenever $x\in O\cup(X\setmin K)$. A policy that maximizes $V_n(\pi, x)$ is defined only on the set $K\setmin O$, and it is left undefined elsewhere. Once the process exits $K\setmin O$ or the stage reaches $n-1$, the task of our control policy is over. Such a deterministic stationary policy (which exists, as demonstrated below) with a measurable selector $f\in\mathbb F$ should be represented as $f^{\tau\mn\tau'} \Let  \underset{\tau\mn\tau'\text{ times}}{\underbrace{(f, f, \ldots, f)}}$ since it is applied only for the first $\tau\mn\tau'$ stages; however, for notational brevity we simply write $f^\infty$ hereafter.


		Quite clearly, letting $n\to\infty$, the monotone convergence theorem gives
		\begin{align*}
			V(\pi, x) & = \lim_{n\to\infty} V_n(\pi, x) = \PP^\pi_x\bigl(\tau < \tau', \tau < \infty\bigr)\\
			& = \EE^\pi_x\Biggl[\sum_{t=1}^{\tau\mn\tau'}\indic{O}(x_t)\Biggr] = \EE^\pi_x\Biggl[\sum_{t=1}^{\tau}\indic{O}(x_{t\mn\tau\mn\tau'})\Biggr].
		\end{align*}
		We note that by definition, the random sum inside the expectation on the right-hand side of the last equality above is the limit of partial (finite) sums, and this ensures that the term inside the expectation is defined on the event $\{\tau\mn\tau' < \infty\}$. By definition note that
		\begin{equation}
		\label{e:Vstar}
			V^\star(x) = \sup_{\pi\in\Pi_M} V(\pi, x) = \sup_{\pi\in\Pi_M}\EE^\pi_x\Biggl[\sum_{t=1}^{\tau\mn\tau'}\indic{O}(x_t)\Biggr].
		\end{equation}

		Consider again the value-iteration functions defined by
		\begin{equation}
		\label{e:valueiter}
			\begin{cases}
				v_0(x) \Let  \indic{O}(x)\\
				v_n(x) \Let  \displaystyle{\indic{O}(x) + \indic{K\setmin O}(x)\max_{a\in A(x)}\int_X Q(\mrm dy|x, a)\indic{K}(y) v_{n-1}(y)}
			\end{cases}
		\end{equation}
		for $x\in X$ and $n\in\N$. The function $v_n$ is clearly identifiable with the optimal value function for the problem of maximizing $\PP^\pi_x\bigl(\tau < \tau', \tau < n\bigr)$ of the process stopped at the $(n-1)$-th stage, $n\in\N$. To get an intuitive idea, fix a deterministic Markov policy $\pi' = (f_t)_{t\in\Nz}$ and take the first iterate $v_0$. (Of course the assumption underlying the notation $(f_t)_{t\in\Nz}$ is that $f_t$ is defined on $\{t < \tau\mn\tau'\}$.) It is immediately clear that the reward at the first step is $1$ if and only if $x\in K$ and $0$ otherwise, and that is precisely $v_0$ irrespective of the policy. For the second iterate note the reward under the policy $\pi'$ is $\indic{O}(x) + \indic{K\setmin O}(x)Q(O|x, f_0(x))$. This is because the reward is $1$ if $x\in O$ and the process terminates at the first stage, or $x\in K\setmin O$ and the reward at the second stage is the probability of hitting $O$ at the second stage. Of course there is no reward if $x\in X\setmin K$. Similarly, for the third iterate the reward is $\indic{O}(x) + \indic{K\setmin O}(x)\int_{K}Q(\mrm d\xi_1|x, f_0(x))\bigl(\indic{O}(\xi_1) + \indic{K\setmin O}(\xi_1)Q(O|\xi_1, f_1(\xi_1))\bigr)$. Note that only those sample paths that stay in $K\setmin O$ at the first step contribute to the reward at the second stage, only those sample paths that stay in $K\setmin O$ for the first and the second stages contribute to the reward at the third stage, and so on.

		\subsection{A general setting and various special cases} 
		\label{s:disc:ss:gensetting}
		Our problem~\eqref{e:problem} can be viewed as a special case of a more general setting. To wit, consider a nonnegative upper semicontinuous reward-per-stage function $r:\mathbb K\lra\posR$ and the problem of maximizing the total reward up to (and including) the hitting time $\tau\mn\tau'$, i.e., maximize $\EE^\pi_x\bigl[\sum_{t=0}^{\tau\mn\tau'} r(x_t, a_t)\bigr]$ over a class of policies. This corresponds to maximization of the reward until exit from the set $K\setmin O$. The value-iteration functions $(v'_n)_{n\in\Nz}$ corresponding to this problem can be written down readily: for $x\in X$ and $n\in\N$ let
		\begin{align*}
			v'_0(x) & \Let  \sup_{a\in A(x)}r(x, a)\indic{O\cup(X\setmin K)}(x),\\
			v'_n(x) & \Let  \sup_{a\in A(x)}\biggl[r(x, a)\indic{O\cup(X\setmin K)}(x) + \indic{K\setmin O}(x)\int_X Q(\mrm dy|x, a)v'_{n-1}(y)\biggr].
		\end{align*}
		Our problem~\eqref{e:problem} corresponds to the case of $r(x, a) = \indic{O}(x)$. Modulo the additional technical complications involving integrability of the value-iteration functions at each stage and the total reward corresponding to initial conditions being well-defined real numbers, the analysis of this more general problem can be carried out in exactly the same way as we do below for the problem~\eqref{e:problem}. While the above more general problem treats both the target set $O$ and the cemetery state $X\setmin K$ equally, the bias towards the target set $O$ is provided in our problem~\eqref{e:problem} by the special structure of the reward $r(x, a) = \indic{O}(x)$.

		From the general framework it is not difficult to arrive at reward-per-stage functions that are meaningful in the context of reachability, avoidance, and safety. For the sake of simplicity, till the end of this subsubsection we suppose that for all initial conditions and admissible policies $\pi\in\Pi$ the stopping times $\tau$ and $\tau'$ are finite $\PP^\pi_x$-almost surely. With this assumption in place, let us look at some examples:

		\begin{itemize}[label=$\circ$, leftmargin=*]

			\item Consider a discounted version of our problem~\eqref{e:problem}, namely, let 
			\[
				V^{(1)}(\pi, x) \Let  \EE^\pi_x\Biggl[\sum_{t=0}^{\tau\mn\tau'} \alpha^t \indic{O}(x_t)\Biggr],
			\]
			where $\alpha\in\;]0, 1[$ is a constant discount factor. From the definitions of $\tau$ and $\tau'$ it follows that $\sum_{t=0}^{\tau\mn\tau'} \alpha^t \indic{O}(x_t) = \alpha^\tau \indic{\{\tau < \tau'\}}$, and in view of the range of $\alpha$ it follows that maximization of $V^{(1)}$ over admissible policies leads to small values of $\tau$ on the set $\{\tau < \tau'\}$ on an average, but it is silent about the values of $\tau$ on $\{\tau > \tau'\}$.
			
			To get a more quantitative idea of the role that the discount factor $\alpha$ plays, let $\tilde\tau$ be a random variable independent of the Markov control process defined in Definition~\ref{d:mcp},\footnote{The random variable $\tilde\tau$ can be defined in a standard way by enlarging the probability space.} with distribution function $\PP(\tilde\tau = n) = (1-\alpha)\alpha^n$ for all $n\in\Nz$. In a standard way we construct the product probability measure $\PP^\pi\otimes\PP$ and denote the expectation with respect to this measure as $\EE^{\pi, \tilde\tau}_x[\cdot]$. We can write 
			\[
				V^{(1)}(\pi, x) = \EE^\pi_x\Biggl[\sum_{t=0}^\infty\alpha^t\indic{O}(x_t)\indic{\{t \le \tau\mn\tau'\}}\Biggr] = (1-\alpha)^{-1}\EE^{\pi, \tilde\tau}_x\bigl[\indic{O}(x_{\tilde\tau})\indic{\{\tilde\tau \le \tau\mn\tau'\}}\bigr].
			\]
			In view of the definitions of $\tau$ and $\tau'$ we get $V^{(1)}(\pi, x) = (1-\alpha)^{-1} \EE^{\pi, \tilde\tau}_x\bigl[\indic{\{\tilde\tau = \tau, \tau < \tau'\}}\bigr]$. This alternative characterization shows that maximization of $V^{(1)}$ over admissible policies leads to smaller values of $\tau$ compared to $\tau'$; moreover, the random variable $\tilde\tau$ gives a quantitative idea of how the choice of $\alpha$ determines the outcome since $\tilde\tau$ is a geometric random variable with parameter $(1-\alpha)$. Choosing a small $\alpha$ implies smaller $\tilde\tau$ with higher probability and may appear to be profitable; however, in certain problems it is possible that the set $O$ may be reachable at $\tilde\tau$ with small probability and the corresponding event of interest $\{\tilde\tau = \tau, \tau < \tau'\}$ may be relatively small for a given initial condition $x$. Moreover, the factor $(1-\alpha)^{-1}$ is small for small values of $\alpha$, and contributes to this phenomenon.

			A second quantitative view of the role of $\alpha$ is offered by the fact that $V^{(1)}(\pi, x) = \EE^{\pi, \tilde\tau}_x\bigl[\sum_{t=0}^{\tilde\tau\mn\tau\mn\tau'}\indic{O}(x_t)\bigr]$. Indeed, we have
			\begin{align*}
				\EE^{\pi, \tilde\tau}_x &\Biggl[\sum_{t=0}^{\tilde\tau\mn\tau\mn\tau'}\indic{O}(x_t)\Biggr] = \EE^{\pi, \tilde\tau}_x\Biggl[\sum_{t=0}^{\tilde\tau}\indic{O}(x_t)\indic{\{t\le\tau\mn\tau'\}}\Biggr]\\
				& = \EE^\pi_x\Biggl[\sum_{n=0}^\infty \alpha^n(1-\alpha)\sum_{t=0}^n \indic{O}(x_t)\indic{\{t\le\tau\mn\tau'\}}\Biggr]\\
				& = \EE^\pi_x\Biggl[\sum_{n=0}^\infty \sum_{t=0}^n \alpha^n\indic{O}(x_t)\indic{\{t\le\tau\mn\tau'\}} - \sum_{n=0}^\infty\sum_{t=0}^n\alpha^{n+1}\indic{O}(x_t)\indic{\{t\le\tau\mn\tau'\}}\Biggr]\\
				& = \EE^\pi_x\Biggl[\sum_{t=0}^\infty \sum_{n=t}^\infty \alpha^n\indic{O}(x_t)\indic{\{t\le\tau\mn\tau'\}} - \sum_{t=0}^\infty\sum_{n=t}^\infty\alpha^{n+1}\indic{O}(x_t)\indic{\{t\le\tau\mn\tau'\}}\Biggr]\\
				& = \EE^\pi_x\Biggl[\sum_{t=0}^\infty \frac{\alpha^t}{1-\alpha}\indic{O}(x_t)\indic{\{t\le\tau\mn\tau'\}} - \sum_{t=0}^\infty\frac{\alpha^{t+1}}{1-\alpha}\indic{O}(x_t)\indic{\{t\le \tau\mn\tau'\}}\Biggr]\\
				& = \EE^\pi_x\Biggl[\sum_{t=0}^\infty\alpha^t\indic{O}(x_t)\indic{\{t\le \tau\mn\tau'\}}\Biggr] = \EE^\pi_x\Biggl[\sum_{t=0}^{\tau\mn\tau'}\alpha^t\indic{O}(x_t)\Biggr] = V^{(1)}(\pi, x).
			\end{align*}
			In this setting we do not have the $(1-\alpha)^{-1}$ factor outside the expectation as in the second version of $V^{(1)}$ above, and it demonstrates that maximizing $V^{(1)}(\pi, x)$ over admissible policies leads to maximizing the probability of the event $\{\tau < \tilde\tau\mn\tau'\}$, where $\alpha$ controls the values of $\tilde\tau$ as before.

			\item Consider the reward-per-stage function $r(x, a) = \indic{O}(x) - \indic{X\setmin O}(x)$. Under integrability assumption on $\tau\mn\tau'$ under all admissible policies, we have 
			\begin{align*}
				V^{(2)}(\pi, x) & \Let  \EE^\pi_x\Biggl[\sum_{t=0}^{\tau\mn\tau'}\bigl(\indic{O}(x_t) - \indic{X\setmin O}(x_t)\bigr)\Biggr]\\
				& = \EE^\pi_x\Biggl[\sum_{t=0}^{\tau\mn\tau'}\bigl(\indic{O}(x_t) - \indic{K\setmin O}(x_t) - \indic{X\setmin K}(x_t)\bigr)\Biggr]\\
				& = \PP^\pi_x(\tau < \tau') - \PP^\pi_x(\tau' < \tau) - \EE^\pi_x[\tau\mn\tau'].
			\end{align*}
			Clearly, maximization of $V^{(2)}$ over admissible policies leads to both the maximal enlargement of the set $\{\tau < \tau'\}$ and minimization of the hitting time $\tau$ on this set.

			\item Consider $r(x, a) = \indic{O}(x) - \indic{X\setmin K}(x)$. This leads to the expected total reward until escape from $K\setmin O$ as\label{page:V3}
			\[
				V^{(3)}(\pi, x) \Let  \EE^\pi_x\Biggl[\sum_{t=0}^{\tau\mn\tau'}\bigl(\indic{O}(x_t) - \indic{X\setmin K}(x_t)\bigr)\Biggr] = \PP^\pi_x(\tau < \tau') - \PP^\pi_x(\tau' < \tau).
			\]
			Since $\PP^\pi_x(\tau < \tau') + \PP^\pi_x(\tau' < \tau) = 1$, maximization of $V^{(3)}$ over admissible policies maximizes the probability of the event $\{\tau < \tau'\}$. Thus, maximizing $V^{(3)}(\pi, x)$ over $\pi\in\Pi$ is a different formulation of the objective of our problem~\eqref{e:problem}. The above analysis also shows that the same objective results if we take the reward-per-stage function to be $\indic{O}(x) - \gamma\indic{X\setmin K}(x)$ for any $\gamma\ge 0$.

			\item Suppose that $\tau\mn\tau'$ is integrable for all admissible policies and consider the reward-per-stage $r(x, a) = \indic{K\setmin O}(x)$. Let 
			\[
				V^{(4)}(\pi, x) \Let  \EE^\pi_x\Biggl[\sum_{t=0}^{\tau\mn\tau'} \indic{K\setmin O}(x_t)\Biggr].
			\]
			Maximization of $V^{(4)}$ over admissible policies leads to large values of $\tau\mn\tau'$ on an average. This is a form of safety problem, the state stays inside $K\setmin O$ for as long as possible on an average.

			\item Suppose that $\tau\mn\tau'$ is integrable for all admissible policies and consider $r(x, a) = \gamma\indic{O}(x) - \indic{K\setmin O}(x)$ for $\gamma \ge 1$. Consider
			\[
				V^{(5)}(\pi, x) \Let  \EE^\pi_x\Biggl[\sum_{t=0}^{\tau\mn\tau'} \bigl(\gamma\indic{O}(x_t) - \indic{K\setmin O}(x_t)\bigr)\Biggr],
			\]
			we see that $V^{(5)}(\pi, x) = \gamma\PP^\pi_x(\tau < \tau') - \EE^\pi_x[\tau\mn\tau']$. We see that maximization of $V^{(5)}$ over admissible policies leads to a balance between maximizing the probability that the state hits the set $O$ before getting out of $K$ and exiting $K$ quickly. This is because it is more profitable to exit from $K$ and get a zero reward than incur negative reward by prolonging the duration of stay in $K\setmin O$. The factor $\gamma$ decides the priorities of the two alternatives. It is trivially clear that $\gamma = 1$ leads to rapid exit from $K$ if the initial condition is in $K\setmin O$.
		\end{itemize}

		Not all the reward-per-stage functions mentioned above can be handled in our present framework. In particular, we make the crucial assumption that the reward-per-stage function is nonnegative, which does not hold in some of the cases above. However, under appropriate growth-rate conditions on the reward-per-stage function, the nonnegativity assumption can be dispensed with.

		In classical finite or infinite-horizon optimal control problems a translation of the (fixed) reward-per-stage function would not change the solution to the problem. However, translations of the reward-per-stage function in random-horizon problems may lead to drastically different policies. We give two examples:
		\begin{itemize}[label=$\circ$, leftmargin=*]
			\item Consider the reward-per-stage functions $r'(x, a) = \indic{O}(x) - \indic{X\setmin K}(x)$ and $r''(x, a) = 2\cdot\indic{O}(x) + \indic{K\setmin O}(x)$; in this case we translate $r'$ on $X$ by $1$, i.e., $r'' = r' + 1$. On the one hand, maximizing $\EE^\pi_x\bigl[\sum_{t=0}^{\tau\mn\tau'} r'(x_t, a_t)\bigr]$ yields a policy that $\PP^\pi_x(\tau < \tau')$ as we have seen before (this is $V^{(3)}$ above). On the other hand, maximizing $\EE^\pi_x\bigl[\sum_{t=0}^{\tau\mn\tau'} r''(x_t, a_t)\bigr]$ yields a policy that tries to keep the state in $K\setmin O$ for as long as possible, and at each stage accrue a reward of $1$, which is certainly better than jumping to $O$ and accruing a reward of $2$ at most.
			\item Consider $r'(x, a) = \indic{O}(x) - \indic{X\setmin K}(x)$ and $r''(x, a) = -\indic{O}(x) - 3\cdot\indic{X\setmin K}(x)$; in this case we translate $r'$ by $-2$ only on its support $O\cup(X\setmin K)$. We have noted above that maximizing $\EE^\pi_x\bigl[\sum_{t=0}^{\tau\mn\tau'}r'(x_t, a_t)\bigr]$ yields a policy that maximizes $\PP^\pi_x(\tau < \tau')$. However, maximizing $\EE^\pi_x\bigl[\sum_{t=0}^{\tau\mn\tau'}r''(x_t, a_t)\bigr]$ yields a policy that tries to keep the state in $K\setmin O$ for the longest possible duration to avoid incurring negative reward.
		\end{itemize}

		\subsection{Further examples}
		For one-dimensional stochastic processes initialized somewhere between two different levels $a$ and $b$, problems such as calculating the probability of hitting the level $a$ before the level $b$ are fairly common, e.g., in random walks, Brownian motion, and diffusions, see, e.g.,~\cite[Chapters~2-3]{ref:peresMC},~\citep{ref:revuzyorCMBM}. It is possible to obtain explicit expressions of these probabilities in a handful of cases.

		Let us consider a controlled Markov chain $(x_t)_{t\in\Nz}$ with a finite state-space $X = \{1, 2, \ldots, m\}$ and a transition probability matrix $Q = [q_{ij}(a)]_{m\times m}$, where $a$ is the action or control variable. Let $O\subsetneqq X$, $K\subsetneqq X$ be subsets of $X$ with $O\subsetneqq K$. Since $X$ is finite, Assumption~\ref{a:key} is satisfied. Consider the problem~\eqref{e:problem} in the context of this Markov chain $(x_t)_{t\in\Nz}$ initialized at some $i_0\in K\setmin O$. By Theorem~\ref{t:exist} the optimal value function $V^\star$ must satisfy the equation
		\begin{align*}
			V^\star(i) & = \indic{O}(i) + \indic{K\setmin O}(i)\max_{a\in A(i)}\sum_{j\in K} q_{ij}(a)V^\star(j)\\
			& = \indic{O}(i) + \indic{K\setmin O}(i)\max_{a\in A(i)}\Biggl(\sum_{j\in O}q_{ij}(a) + \sum_{j\in K\setmin O} q_{ij}(a)V^\star(j)\Biggr)
		\end{align*}
		for all $i\in X$. If the control actions are finite in number, searching for a maximizer over an enumerated list all control actions corresponding to each of the states may be possible if the state and action spaces are not too large. However, the memory requirement for storing such enumerated lists clearly increases exponentially with the dimension of the state and action spaces if the Markov chain is extracted by a discretization procedure based on a grid on the state-space of a discrete-time Markov process evolving, for example, on a subset of Euclidean space. As an alternative, it is possible to search for a maximizer from a parametrized family of functions (vectors) by applying well-known suboptimal control strategies~\cite[Chapter~6]{ref:bertsekasDP2}, \citep{ref:bertsekasNDP, ref:powellADP}. Note that in the case of an uncontrolled Markov chain the equation above reduces to $V^\star(i) = \indic{O}(i) + \indic{K\setmin O}(i)\bigl(\sum_{j\in O} q_{ij} + \sum_{j\in K\setmin O} q_{ij}V^\star(j)\bigr)$, and can be solved as a linear equation on $K\setmin O$ for the vector $V^\star|_{K\setmin O}$. Thus, solving for $V^\star$ yields a method of calculating the probability of hitting $O$ before hitting $X\setmin K$ in uncontrolled Markov chains, and can serve as a verification tool~\citep{ref:kwiatkowskaSMC}.

		In certain cases of uncountable state-space Markov chains the policies and value functions corresponding to maximization of $\PP^\pi_x\bigl(\tau < \tau', \tau < n\bigr)$ can be explicitly calculated for small values of $n$. As an illustration, consider a scalar linear controlled system
		\begin{equation}
		\label{e:linsys}
			x_{t+1} = x_t + a_t + w_t, \quad x_0 = x,\;\; t\in\Nz.
		\end{equation}
		Here $x_t\in\R$ is the state of the system at time $t$, $a_t$ is the action or control at time $t$ taking values in $[-1, 1]$, and $(w_t)_{t\in\Nz}$ is a sequence of independent and identically distributed (i.i.d) standard normal random variables treated as noise inputs to the system. Let us suppose that our target set is $O =\; ]-1, 1[$, safe set is $K = [-3, 3]$, and let us find a greedy policy for our problem, i.e., a policy that maximizes $\PP^\pi_x\bigl(\tau < \tau', \tau < 2\bigr)$. 
		\begin{wrapfigure}{l}[0mm]{0.5\textwidth}
		\begin{center}
			\includegraphics[width=0.48\textwidth]{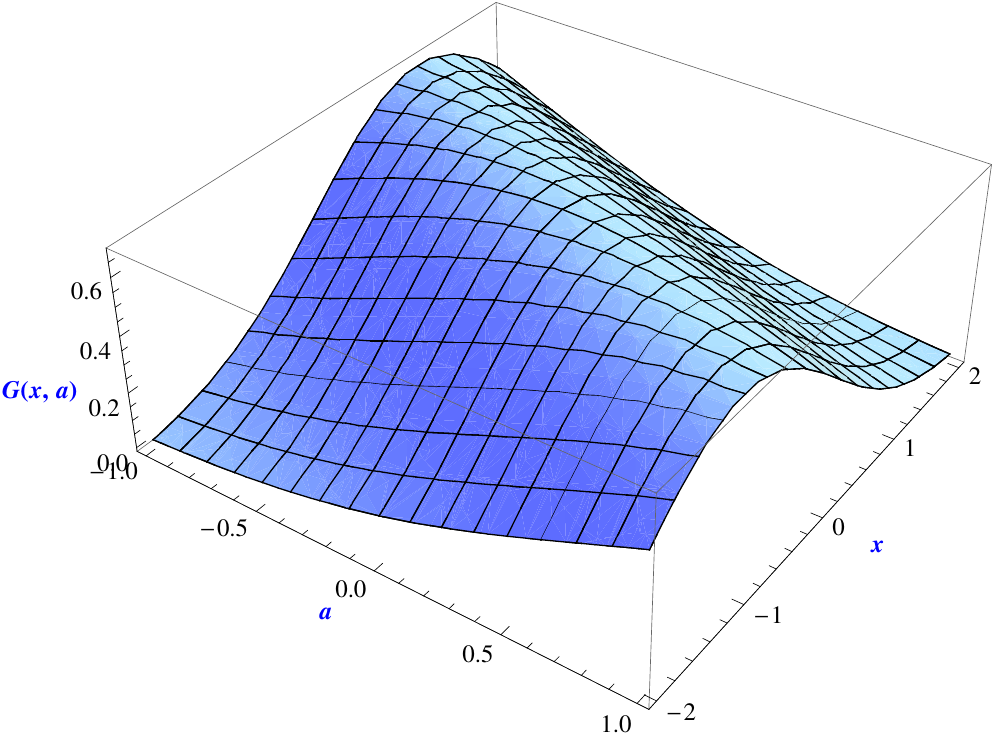}
			\vspace*{-0.2in}
		\end{center}
		\end{wrapfigure}
		The greedy policy tries to maximize $\PP_x\bigl(x_1\in\;]-1, 1[\bigr) = \PP_x\bigl(x + a + w\in\;]-1, 1[\bigr) = \mathfrak N(1-x-a) -\mathfrak N(-1-x-a) =: G(x, a)$, where $\mathfrak N$ is the cumulative distribution function of the standard normal random variable. The function $G$ can be expressed in terms of the complementary error function\footnote{Recall that the complementary error function is defined as $\erfc(r) \Let  \frac{2}{\sqrt \pi}\int_r^\infty \epower{-t^2}\mrm dt = 1-\erf(r)$, where $\erf(\cdot)$ is the standard error function.} as $G(x, a) = \frac{1}{2}\Bigl(\erfc\bigl(-\frac{1}{\sqrt 2}(1-x-a)\bigr) - \erfc\bigl(-\frac{1}{\sqrt 2}(-1-x-a)\bigr)\Bigr)$, and $\argmax_{a\in [-1, 1]}G(x, a)$ can be solved in closed form. Indeed, $\tfrac{\partial G}{\partial a}(x, a) = \frac{1}{\sqrt{2\pi}}\bigl(\epower{-\frac{1}{2}(x+a+1)^2} - \epower{-\frac{1}{2}(x+a-1)^2}\bigr) = 0$ gives $a^\star = f_\star(x) = -x$ as the unconstrained optimizer. Since $a\in[-1, 1]$, we have the constrained maximizer as $f_\star(x) = -\sat(x)$, where $\sat(\cdot)$ is the standard saturation function.\footnote{Recall that the standard saturation function is defined as $\sat(r)$ equals $r$ if $\abs{r} < 1$, $1$ if $r \ge 1$ and $-1$ otherwise.} In other words, we get a bang-bang controller since $x-\sat(x)\neq 0$ on the interior of $K\setmin O$. It is easy to discern the maximizer from the accompanying figure. The corresponding maximal probability is found by substituting the above optimizer back into the dynamic programming equation, and this yields $V_1^\star(x) = \indic{O}(x) + \frac{1}{2}\indic{K\setmin O}(x)\Bigl(\erf\bigl(\frac{1}{\sqrt 2}(x-\sat(x)+1)\bigr) - \erf\bigl(\frac{1}{\sqrt 2}(x-\sat(x) - 1)\bigr)\Bigr)$. For $n = 3$ it turns out that we can no longer compute the optimizer corresponding to the first stage in closed form; the optimizer for the second stage is, of course, $f_\star(x) = -\sat(x)$ calculated above. It is also evident from the accompanying figure that even in this simple example there will arise nontrivial issues with nonconvexity for $n \ge 3$.

		\subsection{Uniqueness of optimal policies}
		So far in our discussion we have not addressed the issue of uniqueness of the optimal policy in our problem~\eqref{e:problem}. (Theorem~\ref{t:exist} shows that an optimal policy exists, so the uniqueness question is meaningful.) It becomes clear from considerations of the geometry of the sets $O$ and $K$ in simple examples that the optimal controller $f_\star$ in Theorem~\ref{t:exist}\emph{\ref{co:exist:2}} is nonunique in general. For instance, consider the linear system considered in~\eqref{e:linsys} above with initial condition $x_0 = 0$, and let $O = \;]-2, -1[\;\cup\;]1, 2[$ and $K = [-3, 3]$. Since the noise is symmetric about the origin, from symmetry considerations it immediately follows that the optimal controller $f_\star$ is nonunique at the origin. Note that $f_\star$ is, of course, defined on $K\setmin O$.

		\subsection{Relation to a probabilistic safety problem}
		Let us digress a little and consider the following probabilistic safety problem: maximize the probability that the state remains inside a safe set $C\subset X$ for $n$ stages, beginning from an initial condition $x\in C$. This, as mentioned earlier, is the probabilistic safety problem addressed in~\citep{ref:prandiniPSafety}. Of course the probability of staying inside $C$ for the first $n$ stages is given by $\PP^\pi_x\bigl(\bigcap_{t=0}^{n-1}\{x_t\in C\}\bigr) = \EE^\pi_x\bigl[\prod_{t=0}^{n-1}\indic{\{x_t\in C\}}\bigr]$. If $\sigma$ is the first exit time from $C$, then $\PP^\pi_x\bigl(\bigcap_{t=0}^{n-1}\{x_t\in C\}\bigr) = \EE^\pi_x\bigl[\prod_{t=0}^{(\sigma-1)\mn(n-1)}\indic{\{x_t\in C\}}\bigr]$. Therefore, in this particular problem there is no difference between the maximal values of $\EE^\pi_x\bigl[\prod_{t=0}^{n-1}\indic{\{x_t\in C\}}\bigr]$ or $\EE^\pi_x\bigl[\prod_{t=0}^{(\sigma-1)\mn(n-1)}\indic{\{x_t\in C\}}\bigr]$. However, the policies arising from the two different maximizations are quite unlike each other. Indeed, whereas the former yields a deterministic Markov policy~\citep{ref:prandiniPSafety} whose every element is defined on all of $X$, the stopping time version yields a deterministic Markov policy whose $t$-th element $\pi_t$ is defined on the set $\{t < \sigma\mn n\}$, just as discussed in paragraph~\ref{pgr:policies}. On the one hand note that the reward in the former case is not affected by further application of the control actions once the state has exited the safe set $C$; the policy resulting from this formulation, however, dictates that the control actions are carried out until (and including) the $(n-2)$-th stage nonetheless. On the other hand, the reward in the latter stopping time version saturates at the stage the state leaves $C$ and future control actions are not defined.

 		It is interesting to note that the Bellman equation developed for probabilistic safety and reachability in~\citep{ref:prandiniPSafety} may be obtained as a special case of~\eqref{e:bellmaneqn} in Theorem~\ref{t:exist} above. This comes as no surprise. The problem of maximizing the probability of staying inside a (measurable) safe set $C\subset X$ for $N$ steps is given by the maximization of $\EE^\pi_x\bigl[\prod_{t=0}^{\sigma\mn(N-1)}\indic{C}(x_t)\bigr]$, where $\sigma$ is the first time to exit $C$ and this clearly translates to minimizing $\PP^\pi_x(\tau < N)$. In our setting, if we let $K$ be the entire state-space $X$, $C = X\setmin O$, and $\tau$ the first time to hit the set $O$, then our problem~\eqref{e:problem} is precisely that in~\citep{ref:prandiniPSafety} with the exception of maximization in place of minimization. It must be mentioned however, that the analysis carried out in~\citep{ref:prandiniPSafety} relies on the approach in~\citep{ref:bertsekasshreve78} and is purely analytical; the strong Feller assumption on the transition kernel in our formulation plays no role there.

	\section{Proofs}
	\label{s:proof}
		This section collects the proofs of the various results in~\secref{s:results}.

		\subsection{Proof of Theorem~\eqref{t:exist}}
		\label{s:mainproof}
		We recall a few standard results about set-valued maps first, followed by sequence of lemmas before getting to the proof of Theorem~\ref{t:exist}. The various definitions in paragraphs~\ref{pgr:policies}, \ref{pgr:recalldef}, and~\ref{pgr:altrep} will be employed without further reference. Just as in~\secref{s:mainres}, for the purposes of this subsection, we let $\Pi_M$ denote the set of admissible Markov policies such that $\pi_t$ is defined on $\mathbb K$ whenever $(\pi_t)_{t\in\Nz}\in\Pi_M$.

		\begin{proposition}[{\cite[Lemma~17.30]{ref:aliprantisIDA}}]
		\label{p:uhccorr}
			Let $\Psi:X\corresp Y$ be an upper hemicontinuous set-valued map between topological spaces with nonempty compact values, and let $f:\mathop{Graph}(\Psi)\lra\R$ be upper semicontinuous.\footnote{Recall that $\mathop{Graph}(\Psi)$ is the set $\bigl\{(x, \Psi(x))\,\big|\, x\in X\bigr\}\subset X\times Y$, the graph of the set-valued map $\Psi$.} Define the function $m:X\lra\R$ by $m(x) \Let  \max_{y\in\Psi(x)} f(x, y)$. Then the function $m$ is upper semicontinuous.
		\end{proposition}

		\begin{proposition}[{\cite[Theorem~18.19]{ref:aliprantisIDA}}]
		\label{p:meascorr}
			Let $X$ be a separable metrizable space and $(S, \Sigma)$ a measurable space. Let $\Psi:S\corresp X$ be a weakly measurable correspondence with nonempty compact values, and suppose $f:S\times X\lra \R$ is a Carath\'eodory function.\footnote{Recall that a Carath\'eodory function $f:S\times X\lra Y$ is a mapping that is measurable in the first variable and continuous in the second, where $(S, \Sigma)$ is a measurable space and $X, Y$ are topological spaces. In particular, if $X$ is a separable and metrizable space, and $Y$ is a metrizable space, every Carath\'eodory function $f:S\times X\lra Y$ is jointly measurable~\cite[Lemma~4.51]{ref:aliprantisIDA}; this is clearly true in the Carath\'eodory functions we consider.} Let us also define the function $m:S\lra\R$ by $m(s) \Let  \max_{x\in\Psi(s)} f(s, x)$, and the correspondence $\mu:S\corresp X$ of maximizers by $\mu(s) \Let  \bigl\{x\in\Psi(s)\,\big|\,f(s, x) = m(s)\bigr\}$. Then the argmax correspondence $\mu$ is measurable and admits a measurable selector.
		\end{proposition}

		\begin{defn}
		{\rm 
			For $u\in b\Borelsigalg{X}^+\cap\bar B$ we define the mapping $Tu$%
			\begin{equation}
			\label{e:dpoperator}
				X\ni x\mapsto Tu(x) \Let  \indic{O}(x) + \indic{K\setmin O}(x)\sup_{a\in A(x)}\int_{K} Q(\mrm dy|x, a) u(y)\in\posR.
			\end{equation}
			The operator $T$ is called the \emph{dynamic programming operator} corresponding to the problem~\eqref{e:problem}.
		}\DefEnd
		\end{defn}

		\begin{lemma}
		\label{l:L0toL0}
			Suppose that Assumption~\eqref{a:key} holds. Then the dynamic programming operator $T$ defined in~\eqref{e:dpoperator} takes $b\Borelsigalg{X}^+\cap\bar B$ into itself. Moreover, there exists a measurable selector $f\in\mathbb F$ such that
			\begin{equation}
			\label{e:finT}
				Tu(x) = \indic{O}(x) + \indic{K\setmin O}(x)\int_{K} Q(\mrm dy|x, f) u(y)\quad \fa x\in X.
			\end{equation}
		\end{lemma}
		\begin{proof}
			Fix $u\in b\Borelsigalg{X}^+\cap\bar B$. Since the transition kernel $Q$ is strongly Feller on $\mathbb K$, the mapping
			\[
				\mathbb K\ni (x, a)\mapsto S(x, a) \Let  \int_X Q(\mrm dy|x, a)\indic{K}(y)u(y)\in\posR
			\]
			is continuous on $\mathbb K$. Also, $S(x, a)$ is bounded whenever $u$ is, a bound of $S$ being the essential supremum norm of $u$. Therefore, since $A(x)$ is compact for each $x\in X$, the function $S^\star(x) \Let  \sup_{a\in A(x)} S(x, a)$ is well-defined on $K\setmin O$, i.e., the sup is attained on $A(x)$ for $x\in K\setmin O$. We also note that since $K\setmin O$ is a measurable set, by Assumption~\ref{a:key}
			\begin{itemize}[label=$\circ$, leftmargin=*, align=right]
				\item the correspondence $K\setmin O\ni x\mapsto A(x)\subset A$ is upper hemicontinuous, and since $S$ is continuous on $\mathbb K$, the map $K\setmin O\ni x\mapsto S^\star(x) \Let  \max_{a\in A(x)}S(x, a)\in\posR$ is an u.s.c.\ function by Proposition~\ref{p:uhccorr};
				\item the correspondence $K\setmin O\ni x\mapsto A(x)\subset A$ is weakly measurable, and since $S$ is continuous on $\mathbb K$ (and therefore is a Carath\'eodory function), there exists a measurable selector $f\in\mathbb F$ such that $S^\star(x) = S(x, f(x))$ for all $x\in K\setmin O$ by Proposition~\ref{p:meascorr}.
			\end{itemize}
			It follows at once that $X\ni x\mapsto Tu(x) = \indic{O}(x) + \indic{K\setmin O}(x) \int_K Q(\mrm dy|x, f(x)) u(y) \in \posR$ is a member of the set $b\Borelsigalg X^+$, and the assertion follows.
		\end{proof}

		\begin{lemma}
		\label{l:dominate}
			Suppose that hypotheses of Theorem~\eqref{t:exist} hold. If $u\in b\Borelsigalg{X}^+\cap\bar B$ satisfies the inequality $u \le Tu$ pointwise on $X$, then also $u \le V^\star$ pointwise on $X$, where $T$ is the dynamic programming operator in~\eqref{e:dpoperator}.
		\end{lemma}
		\begin{proof}
			By definition of $T$ it is clear that we only need to examine the validity of the assertion on $K\setmin O$. Suppose that $u\in b\Borelsigalg{X}^+\cap\bar B$ satisfies the inequality $u\le Tu$ pointwise on $X$. By Lemma~\ref{l:L0toL0} we know that there exists $f\in\mathbb F$ satisfying
			\[
				Tu(x) = \indic{O}(x) + \indic{K\setmin O}(x)\int_K Q(\mrm dy|x, f) u(y)\quad \fa x\in K\setmin O.
			\]
			A straightforward calculation shows that if $u \le Tu$ then $Tu \le T\circ Tu$ on $K\setmin O$. Fix $x\in K\setmin O$. Applying the inequality $u\le Tu$ repeatedly we have
			\begin{align*}
				u(x) & \le \indic{O}(x) + \indic{K\setmin O}(x)\int_K Q(\mrm d\xi_1|x, f) u(\xi_1) \\
				& \le \indic{O}(x) + \indic{K\setmin O}(x)\int_K Q(\mrm d\xi_1|x, f)\biggl[\indic{O}(\xi_1) + \indic{K\setmin O}(\xi_1)\int_K Q(\mrm d\xi_2|\xi_1, f) u(\xi_2)\biggr]\\
				& \cdots
			\end{align*}
			and after $n$ steps
			\begin{align*}
				u(x) & \le \indic{O}(x) + \indic{K\setmin O}(x)\int_K Q(\mrm d\xi_1|x, f)\biggl[\indic{O}(\xi_1) + \ldots \\
				& \qquad\qquad \ldots + \indic{K\setmin O}(\xi_{n-2})\int_K Q(\mrm d\xi_{n-1}|\xi_{n-2}, f)\biggl[\indic{O}(\xi_{n-1})\\
				& \qquad\qquad\qquad\qquad\qquad + \indic{K\setmin O}(\xi_{n-1})\int_K Q(\mrm d\xi_n|\xi_{n-1}, f) u(\xi_{n})\biggr]\cdots\biggr]\\
				& = \Biggl(\indic{O}(x) + \indic{K\setmin O}(x)\int_K Q(\mrm d\xi_1|x, f)\biggl[\indic{O}(\xi_1) + \ldots \\
				& \qquad\qquad \ldots + \indic{K\setmin O}(\xi_{n-2})\int_O Q(\mrm d\xi_{n-1}|\xi_{n-2}, f)\biggr]\Biggr)\\
				& \quad + \Biggl(\indic{K\setmin O}(x)\int_{K\setmin O}Q(\mrm d\xi_1|x, f)\int_{K\setmin O}Q(\mrm d\xi_2|\xi_1, f)\cdots\int_{K}Q(\mrm d\xi_{n}|\xi_{n-1}, f)u(\xi_n)\Biggr).
			\end{align*}
			We claim that the right-hand side of the last equality above is
			\[
				\EE^{f^\infty}_x\Biggl[\sum_{t=0}^{(n-1)\mn\tau\mn\tau'}\indic{O}(x_t)\Biggr] + \EE^{f^\infty}_x\Bigl[\indic{K\setmin O}(x_{(n-1)\mn\tau\mn\tau'})(\indic{K}\cdot u)(x_{n\mn\tau\mn\tau'})\indic{\{\tau\mn\tau' < \infty\}}\Bigr],
			\]
			where $\indic{K}\cdot u(\xi) \Let  \indic{K}(\xi) u(\xi)$ for $\xi\in X$. To see this note that the first term is clear by definition. The second term above is due to the fact that only those trajectories that stay in $K\setmin O$ for $n$ steps (i.e., from stage $0$ through stage $n-1$) contribute to the integrand that features $u$, and this accounts for the factor $\indic{K\setmin O}(x_{(n-1)\mn\tau\mn\tau'})$. Since $\{\tau\mn\tau' < \infty\}$ is a full measure set, the factor $\indic{\{\tau\mn\tau' < \infty\}}$ does not change the value of the integral. Taking the limit of the first term above as $n\to\infty$, the monotone convergence theorem gives
			\[
				\lim_{n\to\infty} \EE^{f^\infty}_x\Biggl[\sum_{t=0}^{(n-1)\mn\tau\mn\tau'} \indic{O}(x_t)\Biggr] = \EE^{f^\infty}_x\Biggl[\sum_{t=0}^{\tau\mn\tau'}\indic{O}(x_t)\Biggr] = V(f^\infty, x) \le V^\star(x),
			\]
			where the last inequality follows from the definition of $V^\star$. Since $u$ is bounded and nonnegative, taking the limit of the second term above as $n\to\infty$, the dominated convergence theorem gives
			\begin{align*}
				\lim_{n\to\infty} & \EE^{f^\infty}_x\Bigl[\indic{K\setmin O}(x_{(n-1)\mn\tau\mn\tau'})(\indic{K}\cdot u)(x_{n\mn\tau\mn\tau'})\indic{\{\tau\mn\tau' < \infty\}}\Bigr]\\
				& = \EE^{f^\infty}_x\Bigl[\indic{K\setmin O}(x_{\tau\mn\tau'})(\indic{K}\cdot u)(x_{\tau\mn\tau'})\indic{\{\tau\mn\tau' < \infty\}}\Bigr]\\
				& = 0
			\end{align*}
			since $\indic{K\setmin O}(x_{\tau\mn\tau'}) = 0$ on the set $\{\tau\mn\tau' < \infty\}$ by definition of the stopping times $\tau$ and $\tau'$. Substituting back we see that $u(x) \le V^\star(x)$, and the assertion follows since $x\in K\setmin O$ is arbitrary.
		\end{proof}

		\begin{lemma}
		\label{l:Vstarsatisfy}
			Suppose that Assumption~\eqref{a:key} holds. Then the value iteration functions $(v_n)_{n\in\Nz}$ defined in~\eqref{e:valueiter} satisfy $v_n\uparrow V^\star$, and the function $V^\star$ satisfies the Bellman equation~\eqref{e:bellmaneqn}.
		\end{lemma}
		\begin{proof}
			From the definition of the value-iteration functions $(v_n)_{n\in\Nz}$ in~\eqref{e:valueiter} we see that $(v_n)_{n\in\Nz}$ is a monotone increasing sequence bounded above by $\indic{X}$. Therefore there exists a measurable function $v^\star:X\lra[0, 1]$ such that $v_n\uparrow v^\star$ pointwise on $X$. By definition of $v_n$ we have
			\[
				\EE^\pi_x\Biggl[\sum_{t=0}^{(n-1)\mn\tau\mn\tau'}\indic{O}(x_t)\Biggr] \le \sup_{\pi\in\Pi_M}\EE^\pi_x\Biggl[\sum_{t=0}^{(n-1)\mn\tau\mn\tau'}\indic{O}(x_t)\Biggr] = v_n(x),
			\]
			and the monotone convergence theorem shows that
			\[
				v^\star(x) = \lim_{n\to\infty}v_n(x) \ge \lim_{n\to\infty}\EE^\pi_x\Biggl[\sum_{t=0}^{(n-1)\mn\tau\mn\tau'}\indic{O}(x_t)\Biggr] = \EE^\pi_x\Biggl[\sum_{t=0}^{\tau\mn\tau'}\indic{O}(x_t)\Biggr] = V(\pi, x).
			\]
			Taking the supremum over $\pi\in\Pi_M$ on the right-hand side shows that $v^\star \ge V^\star$ pointwise on $X$. Note that $v_n|_{O} = 1$ and $v_n|_{X\setmin K} = 0$ for all $n$; therefore $v^\star|_O = 1$ and $v^\star|_{X\setmin K} = 0$.

			Let us define the maps
			\begin{align*}
				\mathbb K\ni (x, a)\mapsto T'v_n(x, a) & \Let  \int_{K}Q(\mrm dy|x, a) v_n(y)\in[0, 1],\\
				\mathbb K\ni (x, a)\mapsto T'v^\star(x, a) & \Let  \int_{K}Q(\mrm dy|x, a) v^\star(y)\in[0, 1].
			\end{align*}
			We note that the transition kernel $Q$ is strongly Feller by Assumption~\ref{a:key}, and therefore $T'v_n, n\in\Nz$ and $T'v^\star$ are continuous functions on $\mathbb K$. Moreover, for all $n\in\Nz$ we define
			\begin{equation}
			\label{e:Tprimeoutside}
			\begin{aligned}
				T'v_n(x, a) & = T'v^\star(x, a) = 1\quad \text{for $x\in O$ and $a\in A(x)$},\\
				T'v_n(x, a) & = T'v^\star(x, a) = 0\quad \text{for $x\in X\setmin K$ and $a\in A(x)$},\\
			\end{aligned}
			\end{equation}
			Since $v_n\uparrow v^\star$ pointwise on $X$, it follows from the definitions above and the monotone convergence theorem that for all $x\in X$ and $a\in A(x)$
			\begin{equation}
			\label{e:Tprimeonannulus}
				T'v_n(x, a)\indic{K\setmin O}(x)\uparrow T'v^\star(x, a)\indic{K\setmin O}(x).
			\end{equation}
			Fix $x\in K\setmin O$. Since $T'v_n$ and $T'v^\star$ are continuous functions on $\mathbb K$, for each $n\in\Nz$ both $\sup_{a\in A(x)}T'v_n(x, a)$ and $\sup_{a\in A(x)}T'v^\star(x, a)$ are attained on $A(x)$. From the definition of $(v_n)_{n\in\Nz}$ in \eqref{e:valueiter} we have $\max_{a\in A(x)}T'v_n(x, a) \le \max_{a\in A(x)}T'v^\star(x, a)$ for all $n\in\Nz$. Also, $\bigl(\max_{a\in A(x)}T'v_n(x, a)\bigr)_{n\in\Nz}$ is a nondecreasing sequence of numbers bounded above by $1$, and therefore it attains a limit. If this limit is strictly less than $\max_{a\in A(x)}T'v^\star(x, a)$, standard easy arguments may be invoked to show that the sequence of continuous functions $\bigl(T'v_n(x, \cdot)\bigr)_{n\in\Nz}$ cannot converge pointwise to $T'v^\star(x, \cdot)$ on $A(x)$, which contradicts~\eqref{e:Tprimeonannulus}. It follows that whenever $x\in K\setmin O$,
			\begin{align*}
				v^\star(x) & = \lim_{n\to\infty} v_n(x) = \lim_{n\to\infty} T v_{n-1}(x)\\
				& = \lim_{n\to\infty} \max_{a\in A(x)} T'v_{n-1}(x, a) = \max_{a\in A(x)} T'v^\star(x, a)\\
				& = Tv^\star(x).
			\end{align*}
			Together with~\eqref{e:Tprimeoutside} this shows that $v^\star$ satisfies the Bellman equation~\eqref{e:bellmaneqn} pointwise on $X$, i.e., $v^\star = Tv^\star$. We have already seen above that $v^\star \ge V^\star$ pointwise on $X$. Since $v^\star = Tv^\star$, the reverse inequality follows from Lemma~\ref{l:dominate}. Therefore, we conclude that $v^\star = V^\star$ identically on $X$.
		\end{proof}

		\begin{lemma}
		\label{l:dsstrategy}
			Let $f^\infty$ be a deterministic stationary policy. Then we have
			\begin{equation}
			\label{e:dsstrategy}
				V(f^\infty, x) = 
				\begin{cases}
					1 & \text{if }x\in O,\\
					\displaystyle{\int_{K} Q(\mrm dy|x, f) V(f^\infty, y)} &  \text{if } x\in K\setmin O,\\
					0 & \text{otherwise}.
				\end{cases}
			\end{equation}
		\end{lemma}
		\begin{proof}
			For $x\in O\cup(X\setmin K)$ the assertions are trivial. Fix $x\in K\setmin O$. From the definition of $V(f^\infty, x)$ we have
			\begin{align*}
				V(f^\infty, x) & = \EE^{f^\infty}\Biggl[\sum_{t=0}^{\tau\mn\tau'}\indic{O}(x_t)\,\Bigg|\,x_0 = x\Biggr]\nonumber\\
				& = \EE^{f^\infty}\Biggl[\indic{O}(x_0)\indic{\{\tau\mn\tau' = 0\}} + \indic{\{\tau\mn\tau'> 0\}}\sum_{t=1}^{\tau\mn\tau'}\indic{O}(x_t)\,\Bigg|\,x_0 = x\Biggr]\nonumber\\
				& = \indic{O}(x) + \EE^{f^\infty}\Biggl[\indic{\{\tau\mn\tau' > 0\}}\sum_{t=1}^{\tau\mn\tau'}\indic{O}(x_t)\,\Bigg|\, x_0 = x\Biggr].
			\end{align*}
			Since $\{\tau\mn\tau' > 0\} = \{x_0 \in K\setmin O\}$ and this event is $\sigalg_0$-measurable, 
			\[
				\EE^{f^\infty}\Biggl[\indic{\{\tau\mn\tau' > 0\}}\sum_{t=1}^{\tau\mn\tau'}\indic{O}(x_t)\,\Bigg|\, x_0 = x\Biggr] = \indic{K\setmin O}(x)\EE^{f^\infty}\Biggl[\sum_{t=1}^{\tau\mn\tau'}\indic{O}(x_t)\,\Bigg|\,x_0 = x\Biggr].
			\]
			Therefore, 
			\begin{align*}
				V(f^\infty, x) & = \indic{O}(x) + \indic{K\setmin O}(x) \EE^{f^\infty}\Biggl[\sum_{t=1}^{\tau\mn\tau'}\indic{O}(x_t)\,\Bigg|\,x_0 = x\Biggr]\\
				& = \indic{O}(x) + \indic{K\setmin O}(x)\EE^{f^\infty}\Biggl[\sum_{t=1}^{\tau}\indic{O}(x_{t\mn\tau\mn\tau'})\,\Bigg|\,x_0 = x\Biggr].
			\end{align*}
			Considering the fact that $V(f^\infty, x) = 0$ for $x\in X\setmin K$ by definition, the Markov property shows that the second term on the right-hand side above equals
			\begin{align*}
				\indic{K\setmin O}(x) &\EE^{f^\infty}\Biggl[\EE^{f^\infty}\Biggl[\sum_{t=1}^{\tau}\indic{O}(x_{t\mn\tau\mn\tau'})\,\bigg|\,x_{1\mn\tau\mn\tau'}\Biggr]\,\Bigg|\,x_0 = x\Biggr]\\
				& = \indic{K\setmin O}(x) \int_K Q(\mrm dy|x, f)\EE^{f^\infty}\Biggl[\sum_{t=1}^{\tau}\indic{O}(x_{t\mn\tau\mn\tau'})\,\Bigg|\,x_{1\mn\tau\mn\tau'} = y\Biggr]\\
				& = \indic{K\setmin O}(x) \int_K Q(\mrm dy|x, f) V(f^\infty, y).
			\end{align*}
			Collecting the above equations we obtain~\eqref{e:dsstrategy}, and this completes the proof.
		\end{proof}

		We are now ready for the proof of the first main result.

		\begin{proof}[Proof of Theorem~\eqref{t:exist}]
			(i) Note that by definition $V^\star$ is nonnegative. The fact that $V^\star$ satisfies the Bellman equation follows from Lemma~\ref{l:Vstarsatisfy}. In view of the definition of $\bar B$ in Theorem~\ref{t:exist} and Lemma~\ref{l:Vstarsatisfy} we conclude that $V^\star$ is minimal in $b\Borelsigalg{X}^+\cap\bar B$ because $u = Tu$ pointwise on $K\setmin O$ implies that $u \le V^\star$ pointwise on $K\setmin O$ for any $u\in b\Borelsigalg{X}^+\cap\bar B$.

			(ii) Lemma~\ref{l:L0toL0} guarantees the existence of a selector $f_\star\in\mathbb F$ such that~\eqref{e:selectorcond} holds. Iterating the equality~\eqref{e:selectorcond} (or~\eqref{e:selectorcondalt}) it follows as in the proof of Lemma~\ref{l:dominate} that for $x\in X$,
			\[
				V^\star(x) = \EE^{f_\star^\infty}_x\Biggl[\sum_{t=0}^{(n-1)\mn\tau\mn\tau'}\indic{O}(x_t)\Biggr] + \EE^{f_\star^\infty}_x\bigl[\indic{K\setmin O}(x_{(n-1)\mn\tau\mn\tau'})(\indic{K} V^\star)(x_{n\mn\tau\mn\tau'})\bigr].
			\]
			Taking limits as $n\to\infty$ on the right, the monotone and dominated convergence theorems give $V^\star(x) = V(f_\star^\infty, x)$. Since $x$ is arbitrary, $V^\star(\cdot) = V(f_\star^\infty, \cdot)$ on $K\setmin O$ and that $f_\star^\infty$ is an optimal policy. Conversely, by Lemma~\ref{l:dsstrategy} it follows that under the stationary deterministic strategy $f_\star^\infty$ we have~\eqref{e:dsstrategy} with $f_\star$ in place of $f$, which is identical to~\eqref{e:selectorcond}.
		\end{proof}

		\subsection{Proofs of the results in~\secref{s:martchar}}
		\label{s:martproofs}
		For the purposes of this subsection we let $\Pi$ denote the set of admissible policies such that $\pi_t$ is defined on $\mathbb K$ whenever $(\pi_t)_{t\in\Nz}\in\Pi$.

		\begin{lemma}
		\label{l:bothsupmart}
			For every policy $\pi\in\Pi$ and initial state $x\in X$ the processes $(\zeta_n)_{n\in\Nz}$ and $\bigl(\indic{K\setmin O}(x_{(n-1)\mn\tau\mn\tau'})(\indic{K}\cdot V^\star)(x_{n\mn\tau\mn\tau'})\bigr)_{n\in\Nz}$ are both nonnegative $(\sigalg_n)_{n\in\Nz}$- supermartingales under $\PP^\pi_x$.
		\end{lemma}
		\begin{proof}
			It is clear that both processes are nonnegative and $(\sigalg_n)_{n\in\Nz}$-adapted. Fix $n\in\N$, an initial state $x\in X$, a policy $\pi\in\Pi$, and on the event $\{\tau\mn\tau' > n\}$ fix a history $h_{n} = \bigl(x, a_0, x_1, a_1, \ldots, x_{n-1}, a_{n-1}, x_{n}\bigr)$. Let $a_{n} \Let  \pi_n(h_n)$ on $\{\tau\mn\tau' > n\}$. Then
			\begin{align*}
				\zeta_{n+1} & = W_{n+1}(\pi, x) + \indic{K\setmin O}(x_{n\mn\tau\mn\tau'})(\indic{K} V^\star)(x_{(n+1)\mn\tau\mn\tau'})\\
					& = W_n(\pi, x) + \indic{O}(x_{n\mn\tau\mn\tau'})\indic{\{\tau\mn\tau' = n\}} + \indic{K\setmin O}(x_{n\mn\tau\mn\tau'})(\indic{K} V^\star)(x_{(n+1)\mn\tau\mn\tau'})\\
					& = W_n(\pi, x) + \indic{\{\tau\mn\tau' = n\}}\indic{O}(x_{n\mn\tau\mn\tau'}) + \indic{\{\tau\mn\tau' > n\}}(\indic{K} V^\star)(x_{(n+1)\mn\tau\mn\tau'}).
			\end{align*}
			Since $\{x_{n\mn\tau\mn\tau'} \in O\} \subset \{\tau\mn\tau' = n\}$, we have
			\begin{align*}
				\indic{\{\tau\mn\tau' = n\}} & \indic{O}(x_{n\mn\tau\mn\tau'}) + \indic{\{\tau\mn\tau' > n\}}(\indic{K} V^\star)(x_{(n+1)\mn\tau\mn\tau'})\\
				& = \indic{\{\tau\mn\tau' \ge n\}}\bigl(\indic{O}(x_{n\mn\tau\mn\tau'}) + \indic{K\setmin O}(x_{n\mn\tau\mn\tau'})(\indic{K} V^\star)(x_{(n+1)\mn\tau\mn\tau'})\bigr).
			\end{align*}
			Since $\{\tau\mn\tau' \ge n\} = \{\tau\mn\tau' > n-1\} = \{x_{(n-1)\mn\tau\mn\tau'}\in K\setmin O\}$, it follows that
			\begin{align*}
				\zeta_{n+1} = W_n(\pi, x) & + \indic{K\setmin O}(x_{(n-1)\mn\tau\mn\tau'})\cdot\\
				& \bigl(\indic{O}(x_{n\mn\tau\mn\tau'}) + \indic{K\setmin O}(x_{n\mn\tau\mn\tau'})(\indic{K}V^\star)(x_{(n+1)\mn\tau\mn\tau'})\bigr).
			\end{align*}
			Therefore, keeping in mind the definition of $a_n$ above,
			\begin{align}
			\label{e:keymartineq}
				\EE^\pi_x\bigl[\zeta_{n+1}\,\big|\,\sigalg_{n\mn\tau\mn\tau'}\bigr] & = W_n(\pi, x) + \indic{K\setmin O}(x_{(n-1)\mn\tau\mn\tau'}) T'V^\star(x_{n\mn\tau\mn\tau'}, a_n)\nonumber\\
				& \le W_n(\pi, x) + \indic{K\setmin O}(x_{(n-1)\mn\tau\mn\tau'})V^\star(x_{n\mn\tau\mn\tau'})\\
				& = \zeta_n,\nonumber
			\end{align}
			where the inequality holds $\PP^\pi_x$-almost surely. Therefore, the process $(\zeta_n)_{n\in\Nz}$ is a nonnegative $(\sigalg_{n\mn\tau\mn\tau'})_{n\in\Nz}$- supermartingale, and hence also a $(\sigalg_n)_{n\in\Nz}$- supermartingale. Considering that the sequence $\bigl(W_n(\pi, x)\bigr)_{n\in\Nz}$ is nondecreasing, from the definitions in~\eqref{e:zetadef} and the fact that the process $(\zeta_n)_{n\in\Nz}$ is a $(\sigalg_n)_{n\in\Nz}$- supermartingale we see that the process $\bigl(\indic{K\setmin O}(x_{(n-1)\mn\tau\mn\tau'})(\indic{K}V^\star)(x_{n\mn\tau\mn\tau'})\bigr)_{n\in\Nz}$ is also a $(\sigalg_n)_{n\in\Nz}$- supermartingale under $\PP^\pi_x$.
		\end{proof}


		\begin{proof}[Proof of Theorem~\eqref{t:martcharpolicy}]
			Lemma~\ref{l:bothsupmart} confirms that both of the two adapted processes $(\zeta_n)_{n\in\Nz}$ and $\bigl(\indic{K\setmin O}(x_{(n-1)\mn\tau\mn\tau'})(\indic{K}V^\star)(x_{n\mn\tau\mn\tau'})\bigr)_{n\in\N}$ converge almost surely and are nonincreasing in expectation, both under $\PP^\pi_x$. Let $\Lambda^\pi(x) \Let  \lim_{n\to\infty}\EE^\pi_x[\zeta_n]$. We then have
			\begin{align}
			\label{e:thriftycalc}
				V^\star(x) & = \EE^\pi_x[\zeta_0] \ge \lim_{n\to\infty}\EE^\pi_x[\zeta_n]\nonumber\\
				& = \lim_{n\to\infty}\Bigl(\EE^\pi_x\bigl[W_n(\pi, x)\bigr] + \EE^\pi_x\bigl[\indic{K\setmin O}(x_{(n-1)\mn\tau\mn\tau'})(\indic{K}V^\star)(x_{n\mn\tau\mn\tau'})\bigr]\Bigr)\\
				& \ge V(\pi, x).\nonumber
			\end{align}
			The assertion is now an immediate consequence of~\eqref{e:thriftycalc}.
		\end{proof}


		\begin{proof}[Proof of Theorem~\eqref{t:thriftychar}]
			Suppose that (i) holds. Since $\EE^\pi_x[\zeta_n]$ is nonincreasing with $n$ it follows that $\EE^\pi_x[\zeta_{n+1}] = \EE^\pi_x[\zeta_n] = \ldots = \EE^\pi_x[\zeta_0] = V^\star(x)$ for every $n\in\N$. Therefore, equality must hold $\PP^\pi_x$-almost surely in~\eqref{e:keymartineq}, and (ii) follows.

			Suppose that (ii) holds. Then equality holds in~\eqref{e:keymartineq} almost surely under $\PP^\pi_x$, and therefore $\PP^\pi_x$-almost everywhere on the set $\{x_{n\mn\tau\mn\tau'}\in K\setmin O\} = \{\tau\mn\tau' > n\}$ we have $T'V^\star(x_n, a_n) = V^\star(x_n)$, and (iii) follows.

			Suppose that (iii) holds. Then taking expectations in~\eqref{e:keymartineq} we arrive at $\EE^\pi_x[\zeta_{n+1}] = \EE^\pi_x[\zeta_n] = \ldots = \EE^\pi_x[\zeta_0] = V^\star(x)$. As a result we have $\Lambda^\pi(x) = V^\star(x)$, and (i) follows.
		\end{proof}


		\begin{proof}[Proof of Theorem~\eqref{t:Vprimechar}]
			It follows readily from the definition of the stopping times $\tau$ and $\tau'$ that the process $(\zeta'_n)_{n\in\Nz}$ defined in~\eqref{e:zetapdef} is a bounded process, and by assumption it is a $(\sigalg_n)_{n\in\Nz}$ -martingale under $\PP^{\pi^\star}_x$. Doob's Optional Sampling Theorem~\cite[Theorem~2, p.~422]{ref:raoProbTheo} applied to $(\zeta'_n)_{n\in\Nz}$ at the stopping time $\tau\mn\tau'$ gives us
			\[
				\EE^{\pi^\star}_x\bigl[\zeta'_{\tau\mn\tau'}\bigr] = \EE^{\pi^\star}_x\bigl[\zeta'_0\bigr] = V'(x),
			\]
			where the last equality follows from the definition of $\zeta'_0$. From~\eqref{e:zetadef} we get
			\begin{align*}
				\EE^{\pi^\star}_x\bigl[\zeta'_{\tau\mn\tau'}\bigr] & = \EE^{\pi^\star}_x\Bigl[W_{\tau\mn\tau'-1}(\pi^\star, x) + \indic{K\setmin O}(x_{\tau\mn\tau'-1})\bigl(\indic{K}\cdot V'\bigr)(x_{\tau\mn\tau'})\Bigr]\\
				& = \EE^{\pi^\star}_x\Biggl[\sum_{t=0}^{\tau\mn\tau'-1} \indic{O}(x_t) + \indic{K\setmin O}(x_{\tau\mn\tau'-1})\bigl(\indic{K}\cdot V'\bigr)(x_{\tau\mn\tau'})\Biggr]\\
				& = \EE^{\pi^\star}_x\Bigl[\indic{K\setmin O}(x_{\tau\mn\tau'-1})\bigl(\indic{K}\cdot V'\bigr)(x_{\tau\mn\tau'})\Bigr].
			\end{align*}
			By definition of $\tau$ and $\tau'$, $\indic{K\setmin O}(x_{\tau\mn\tau'-1})$ equals $1$ on $\{\tau\mn\tau' < \infty\}$, and by our hypotheses the set $\{\tau\mn\tau' < \infty\}$ is a $\PP^{\pi^\star}_x$-full-measure set. Continuing from the last equality above we arrive at
			\begin{align}
				\EE^{\pi^\star}_x\bigl[\zeta'_{\tau\mn\tau'}\bigr] & = \EE^{\pi^\star}_x\Bigl[\indic{\{\tau\mn\tau' < \infty\}}\bigl(\indic{K}\cdot V'\bigr)(x_{\tau\mn\tau'})\Bigr]\nonumber\\
				& = \EE^{\pi^\star}_x\bigl[\indic{\{\tau\mn\tau' < \infty\}}\bigl(\indic{\{\tau < \tau'\}}\indic{K}(x_\tau)V'(x_{\tau}) + \indic{\{\tau > \tau'\}}\indic{K}(x_{\tau'})V'(x_{\tau'})\bigr)\bigr]\nonumber\\
				& = \EE^{\pi^\star}_x\bigl[\indic{\{\tau\mn\tau' < \infty\}}\indic{\{\tau < \tau'\}}\bigr]\label{e:bydef}\\
				& = \PP^{\pi^\star}_x\bigl(\tau < \tau', \tau < \infty\bigr),\nonumber
			\end{align}
			where the equality in~\eqref{e:bydef} follows from the assumptions on $V'$ and the definitions of $\tau$ and $\tau'$. Collecting the equations above we get $V'(x) = \PP^{\pi^\star}_x\bigl(\tau < \tau', \tau < \infty\bigr)$ as asserted.
		\end{proof}


		It is of interest to note that the hypotheses of Theorem~\ref{t:Vprimechar} requires at least one of the stopping times $\tau$ or $\tau'$ to be finite. Let us examine the case of $\tau\mn\tau'$ being $\infty$ on a set of positive probability. Following the proof of Theorem~\ref{t:Vprimechar}, we see that in this case we have to agree on the value of $V'(x_{\tau\mn\tau'})$ on $\{\tau\mn\tau' = \infty\}$. If $\lim_{n\to\infty} V'_n(\pi^\star, x)$ exists, then we can always let $V'(x_{\tau\mn\tau'})$ take this value on the set $\{\tau\mn\tau' = \infty\}$. However, the context of the problem offers another alternative, namely, to set $V'(x_{\tau\mn\tau'}) = 0$ on $\{\tau\mn\tau' = \infty\}$. This is because if $x_t\in K\setmin O$ for all $t\in\Nz$, then the value of $x_{\tau\mn\tau'}$ is of no consequence at all.

	\section{Conclusions and Future Work}
	\label{s:concl}
		The purpose of this article was to present a dynamic programming based solution to the problem of maximizing the probability of attaining a target set before hitting a cemetery set, and furnish an alternative martingale characterization of optimality in terms of thrifty and equalizing policies. Several related problems of interest were sketched in~\secref{s:disc:ss:gensetting}. Some of these problems do not admit an immediate solution in the dynamic programming framework we established here because of our central assumption that the cost-per-stage function is nonnegative. This issue deserves further investigation.

		The results in this article also provide clear indications to the possibility of developing verification tools for probabilistic computation tree logic~\citep{ref:kwiatkowskaSMC} in terms of dynamic programming operators. This matter is under investigation and will be reported in~\citep{ref:ramponiPCTL}. Implementation of the dynamic-programming algorithm in this article is challenging due to integration over subsets of the state-space, and suboptimal policies are needed. In this context development of a possible connection with `greedy-time-optimal' policies~\cite[Chapters~4, 7]{ref:meynCTCN}, originally proposed as a tractable alternative to optimal policies in demand-driven large-scale production systems, is being sought.

	\section*{Acknowledgement}
		The authors thank On\'esimo Hern\'andez-Lerma for helpful suggestions and pointers to references, and Sean Summers for posing the problem.

\bigskip

\def\cprime{$'$}

\bigskip
\bigskip


\begin{thebibliography}{38}
\expandafter\ifx\csname natexlab\endcsname\relax\def\natexlab#1{#1}\fi
\expandafter\ifx\csname url\endcsname\relax
  \def\url#1{\texttt{#1}}\fi
\expandafter\ifx\csname urlprefix\endcsname\relax\def\urlprefix{URL }\fi

\bibitem[{Abate et~al.(2008)Abate, Prandini, Lygeros, and
  Sastry}]{ref:prandiniPSafety}
Abate, A., Prandini, M., Lygeros, J., Sastry, S., 2008. Probabilistic
  reachability and safety for controlled discrete time stochastic hybrid
  systems. Automatica 44~(11), 2724--2734.

\bibitem[{Aliprantis and Border(2006)}]{ref:aliprantisIDA}
Aliprantis, C., Border, K.~C., 2006. Infinite {D}imensional {A}nalysis: a
  {H}itchhiker's {G}uide, 3rd Edition. Springer-Verlag, Berlin.

\bibitem[{Bertsekas and Tsitsiklis(1996)}]{ref:bertsekasNDP}
Bertsekas, D., Tsitsiklis, J., 1996. Neuro-{D}ynamic {P}rogramming. Athena
  Scientific.

\bibitem[{Bertsekas(2007)}]{ref:bertsekasDP2}
Bertsekas, D.~P., 2007. Dynamic {P}rogramming and {O}ptimal {C}ontrol, 3rd
  Edition. Vol.~2. Athena Scientific.

\bibitem[{Bertsekas and Shreve(1978)}]{ref:bertsekasshreve78}
Bertsekas, D.~P., Shreve, S.~E., 1978. Stochastic {O}ptimal {C}ontrol: the
  {D}iscrete-{T}ime {C}ase. Vol. 139 of Mathematics in Science and Engineering.
  Academic Press Inc. [Harcourt Brace Jovanovich Publishers], New York.

\bibitem[{Boda et~al.(2004)Boda, Filar, Lin, and Spanjers}]{ref:boda04}
Boda, K., Filar, J.~A., Lin, Y., Spanjers, L., 2004. Stochastic target hitting
  time and the problem of early retirement. IEEE Transactions on Automatic
  Control 49~(3), 409--419.

\bibitem[{Borkar(1988)}]{ref:borkarConvexAnalyticApproach}
Borkar, V.~S., 1988. A convex analytic approach to {M}arkov decision processes.
  Probabability Theory and Related Fields 78~(4), 583--602.

\bibitem[{Borkar(1991)}]{ref:borkarTopicsControlledMC}
Borkar, V.~S., 1991. Topics in {C}ontrolled {M}arkov {C}hains. Vol. 240 of
  Pitman Research Notes in Mathematics Series. Longman Scientific \& Technical,
  Harlow.

\bibitem[{Bouakiz and Kebir(1995)}]{ref:bouakizTargetHitting}
Bouakiz, M., Kebir, Y., 1995. Target-level criterion in markov decision
  processes. Journal of Optimization Theory and Applications 86~(1), 1--15.

\bibitem[{Chatterjee et~al.(2008)Chatterjee, Cinquemani, Chaloulos, and
  Lygeros}]{ref:recstrat}
Chatterjee, D., Cinquemani, E., Chaloulos, G., Lygeros, J., 2008. Stochastic
  control up to a hitting time: optimality and rolling-horizon implementation.
  \url{http://arxiv.org/abs/0806.3008}.

\bibitem[{Derman(1970)}]{ref:dermanMDP}
Derman, C., 1970. Finite {S}tate {M}arkovian {D}ecision {P}rocesses. Vol.~67 of
  Mathematics in Science and Engineering. Academic Press, New York.

\bibitem[{Diga{\u\i}lova and
  Kurzhanski{\u\i}(2004)}]{ref:kurzhanskiiAttainability}
Diga{\u\i}lova, I.~A., Kurzhanski{\u\i}, A.~B., 2004. The attainability problem
  under stochastic perturbations. Differentsial\cprime nye Uravneniya 40~(11),
  1494--1499, 1582.

\bibitem[{Dubins and Savage(1976)}]{ref:dubinsHowToGamble}
Dubins, L.~E., Savage, L.~J., 1976. Inequalities for {S}tochastic {P}rocesses
  ({H}ow to {G}amble if {Y}ou {M}ust). Dover Publications Inc., New York,
  corrected republication of the 1965 edition.

\bibitem[{Dynkin(1963)}]{ref:dynkinOptStopping}
Dynkin, E.~B., 1963. Optimum choice of the stopping moment of a {M}arkov
  process. Doklady Academii Nauk SSR 150, 238--240.

\bibitem[{Eaton and Zadeh(1962)}]{ref:eatonzadeh62}
Eaton, J.~H., Zadeh, L.~A., 1962. Optimal pursuit strategies in discrete-state
  probabilistic systems. Transactions of the ASME Ser. D. J. Basic Engineering
  84, 23--29.

\bibitem[{Gao et~al.(2007)Gao, Lygeros, and
  Quincampoix}]{ref:lygerosUncertainHyb}
Gao, Y., Lygeros, J., Quincampoix, M., 2007. On the reachability problem for
  uncertain hybrid systems. IEEE Transactions on Automatic Control 52~(9),
  1572--1586.

\bibitem[{Hern{\'a}ndez-Lerma and Lasserre(1996)}]{ref:hernandez-lerma1}
Hern{\'a}ndez-Lerma, O., Lasserre, J.~B., 1996. Discrete-{T}ime {M}arkov
  {C}ontrol {P}rocesses: {B}asic {O}ptimality {C}riteria. Vol.~30 of
  Applications of Mathematics. Springer-Verlag, New York.

\bibitem[{Hern{\'a}ndez-Lerma and Lasserre(1999)}]{ref:hernandez-lerma2}
Hern{\'a}ndez-Lerma, O., Lasserre, J.~B., 1999. Further {T}opics on
  {D}iscrete-{T}ime {M}arkov {C}ontrol {P}rocesses. Vol.~42 of Applications of
  Mathematics. Springer-Verlag, New York.

\bibitem[{Karatzas and Sudderth(2009)}]{ref:karatzasDP}
Karatzas, I., Sudderth, W., 2009. Two characterizations of optimality in
  dynamic programming. Applied Mathematics \& Optimization;
  \url{http://www.springerlink.com/content/340m82862p817446/?p=f54a99eb8ef3432%
cb45724c3d5ee8baa&pi=0}.

\bibitem[{Kesten and Spitzer(1975)}]{ref:kestenMCP}
Kesten, H., Spitzer, F., 1975. Controlled {M}arkov chains. Annals of
  Probability 3, 32--40.

\bibitem[{Kushner(1971)}]{ref:kushnerIntroStochControl}
Kushner, H., 1971. Introduction to {S}tochastic {C}ontrol. Holt, Rinehart and
  Winston, Inc., New York.

\bibitem[{Kwiatkowska et~al.(2007)Kwiatkowska, Norman, and
  Parker}]{ref:kwiatkowskaSMC}
Kwiatkowska, M., Norman, G., Parker, D., 2007. Stochastic model checking. In:
  Lecture Notes in Comuter Science. Vol. 4486. Springer-Verlag.

\bibitem[{Levin et~al.(2009)Levin, Peres, and Wilmer}]{ref:peresMC}
Levin, D.~A., Peres, Y., Wilmer, E.~L., 2009. Markov {C}hains and {M}ixing
  {T}imes. American Mathematical Society, USA, with an Appendix written by
  James G. Propp and David B. Wilson.

\bibitem[{Meyn(2008)}]{ref:meynCTCN}
Meyn, S.~P., 2008. Control {T}echniques for {C}omplex {N}etworks. Cambridge
  University Press, Cambridge.

\bibitem[{Ohtsubo(2003)}]{ref:ohtsuboValueIter}
Ohtsubo, Y., 2003. Value iteration methods in risk minimizing stopping
  problems. Journal of Computational and Applied Mathematics 152~(1-2),
  427--439.

\bibitem[{Peskir and Shiryaev(2006)}]{ref:shiryaevOptStop}
Peskir, G., Shiryaev, A.~N., 2006. Optimal {S}topping and {F}ree-{B}oundary
  {P}roblems. Lectures in Mathematics ETH {Z\"u}rich. Birkh\"auser Verlag,
  Basel.

\bibitem[{Powell(2007)}]{ref:powellADP}
Powell, W.~B., 2007. Approximate {D}ynamic {P}rogramming. Wiley Series in
  Probability and Statistics. Wiley-Interscience [John Wiley \& Sons], Hoboken,
  NJ.

\bibitem[{Prajna et~al.(2007)Prajna, Jadbabaie, and Pappas}]{ref:prajna07}
Prajna, S., Jadbabaie, A., Pappas, G.~J., 2007. A framework for worst-case and
  stochastic safety verification using barrier certificates. IEEE Transactions
  on Automatic Control 52~(8), 1415--1428.

\bibitem[{Prandini and Hu(2006)}]{ref:prandini06}
Prandini, M., Hu, J., 2006. A stochastic approximation method for reachability
  computations. In: Stochastic Hybrid Systems. Vol. 337 of Lecture Notes in
  Control and Information Sciences. Springer, Berlin, pp. 107--139.

\bibitem[{Ramponi et~al.(2009)Ramponi, Chatterjee, Summers, and
  Lygeros}]{ref:ramponiPCTL}
Ramponi, F., Chatterjee, D., Summers, S., Lygeros, J., 2009. On the connections
  between {PCTL} and {D}ynamic {P}rogramming.
  \url{http://arxiv.org/abs/0910.4738}.

\bibitem[{Rao and Swift(2006)}]{ref:raoProbTheo}
Rao, M.~M., Swift, R.~J., 2006. Probability {T}heory with {A}pplications, 2nd
  Edition. Vol. 582 of Mathematics and Its Applications. Springer-Verlag.

\bibitem[{Revuz and Yor(1999)}]{ref:revuzyorCMBM}
Revuz, D., Yor, M., 1999. Continuous {M}artingales and {B}rownian {M}otion, 3rd
  Edition. Vol. 293 of Grundlehren der Mathematischen Wissenschaften.
  Springer-Verlag, Berlin.

\bibitem[{Schmidli(2008)}]{ref:schmidliInsurance}
Schmidli, H., 2008. Stochastic {C}ontrol in {I}nsurance. Probability and its
  Applications. Springer-Verlag London Ltd., London.

\bibitem[{Simon(1957)}]{ref:simonModelsofMan}
Simon, H.~A., 1957. Models of man, social and rational. {M}athematical essays
  on rational human behavior in a social setting. John Wiley \& Sons Inc., New
  York.

\bibitem[{Tomlin et~al.(2000)Tomlin, Lygeros, and Sastry}]{ref:tomlin2000}
Tomlin, C.~J., Lygeros, J., Sastry, S., 2000. {A game theoretic approach to
  controller design for hybrid systems}. Proceedings of IEEE 88, 949--969.

\bibitem[{Watkins and Lygeros(2003)}]{ref:lygeros03}
Watkins, O., Lygeros, J., 2003. Stochastic reachability for discrete-time
  systems: an application to aircraft collision avoidance. In: 42nd IEEE
  Conference on Decision and Control. Vol.~5. pp. 5314--5319.

\bibitem[{Whittle(1983)}]{ref:whittleOptimization}
Whittle, P., 1983. Optimization {O}ver {T}ime. Vol.~II of Wiley Series in
  Probability and Mathematical Statistics: Applied Probability and Statistics.
  John Wiley \& Sons Ltd., Chichester.

\bibitem[{Zhu and Guo(2006)}]{ref:guoSemimartingaleMCP}
Zhu, Q., Guo, X., 2006. A semimartingale characterization of average optimal
  stationary policies for {M}arkov decision processes. Journal of Applied
  Mathematics and Stochastic Analysis, Art.\ ID
  81593.\url{http://www.hindawi.com/GetArticle.aspx?doi=10.1155/JAMSA/2006/815%
93}.

\end{thebibliography}
\end{document}